\documentclass[12pt]{article}

\usepackage[margin=1in]{geometry}

\usepackage{amsmath}
\usepackage{amsthm}
\usepackage{amsfonts}
\usepackage{multicol}
\usepackage[numbers,sort&compress]{natbib}
\usepackage{booktabs}
\usepackage{titlesec}
\usepackage{longtable,tabu}

\usepackage{enumitem}
\setlist[enumerate]{label={\upshape(\alph*)}}

\usepackage{color}

\usepackage{tikz}
\usetikzlibrary{calc}
\tikzstyle{vertex}=[circle, draw, inner sep=0pt, minimum size=4pt,fill=black]

\tikzstyle{hollowvertex}=[circle, draw, inner sep=0pt, minimum size=4pt]

\tikzstyle{namedvertex}=[circle, draw, inner sep=1pt, minimum size=12pt]

\tikzstyle{phantomvertex}=[circle, draw, inner sep=0pt, minimum size=4pt,color=white]

\usetikzlibrary{decorations.pathreplacing}
\usepackage{subfigure}
\usepackage{tkz-graph}

\newcommand{\nrel}{\operatorname{NRel}}
\newcommand{\ncrel}{\operatorname{NCRel}}
\newcommand{\cw}{\operatorname{CW}}
\newcommand{\cs}{\operatorname{CS}}
\newtheorem{theorem}{Theorem}[section]
\newtheorem{lemma}[theorem]{Lemma}

\newtheorem{corollary}[theorem]{Corollary}

\theoremstyle{definition}

\newtheorem{conjecture}[theorem]{Conjecture}
\newtheorem{problem}{Problem}

\begin{document}

\title{The node cop-win reliability of unicyclic and bicyclic graphs}

\author{Maimoonah Ahmed\\
\small Department of Mathematics \& Statistics\\
\small University of Guelph\\
\small maimoona@uoguelph.ca\\
\and
Ben Cameron\\
\small School of Computer Science\\
\small University of Guelph\\
\small ben.cameron@uoguelph.ca\\

}
\date{\today}

\maketitle

\begin{abstract}
Various models to quantify the reliability of a network have been studied where certain components of the graph may fail at random and the probability that the remaining graph is connected is the proxy for reliability. In this work we introduce a strengthening of one of these models by considering the probability that the remaining graph is not just connected but also cop-win. A graph is \textit{cop-win} if one cop can guarantee capture of a fleeing robber in the well-studied pursuit-evasion game of Cops and Robber. More precisely, for a graph $G$ with nodes that are operational independently with probability $p$ and edges that are operational if and only if both of their endpoints are operational, the \textit{node cop-win reliability} of $G$, denoted $\ncrel(G,p)$, is the probability that the operational nodes induce a cop-win subgraph of $G$. It is then of interest to find graphs $G$ with $n$ nodes and $m$ edges such that $\ncrel(G,p)\ge\ncrel(H,p)$ for all $p\in[0,1]$ and all graphs $H$ with $n$ nodes and $m$ edges. Such a graph is called \textit{uniformly most reliable}. We show that uniformly most reliable graphs exist for unicyclic and bicyclic graphs, respectively. This is in contrast to the fact that there are no known sparse graphs maximizing the corresponding notion of node reliability.

\end{abstract}

\section{Introduction}
Graphs often model real-world networks whose reliability is of paramount importance. To this end there have been many models proposed to quantify the reliability of a network. One of particular interest pioneered by Moore and Shannon \cite{MooreShannon1956} is to assume that various components fail independently with probability $p$ and the reliability of the graph is the probability that the graph is connected. While Moore and Shannon originally considered edge failures in their model, in many situations it is more natural to consider a model where the nodes can fail \cite{StivarosThesis, BrownMol2016, twoterminalnodereliability}.  
To model this situation mathematically, Stivaros \cite{StivarosThesis} introduced the following. Let $G$ be a graph whose nodes are operational independently with probability $p$, for some $p\in[0,1]$ and whose edges are operational if and only if both their endpoints are operational. A \textit{connected set} of $G$ is a subset $S$ of $V(G)$ such that the subgraph induced by $S$ is connected. The \textit{node reliability} of $G$, denoted $\nrel(G,p)$, is defined to be the probability that the operational nodes of $G$ form a connected set. Let $S_i(G)$ denote the number of connected sets of $G$ of order $i$. Then we can express $\nrel(G,p)$ as the univariate polynomial,
$$
\nrel(G,p)=\sum\limits_{i=1}^n S_i(G)(1-p)^{n-i}p^i.
$$

When designing a network, it is desirable to do so in a way that it is most likely that the remaining nodes can communicate with each other even after nodes fail, i.e. to design a graph that maximizes $\nrel(G,p)$. However, there is some difficulty in determining which values of $p$ we wish the network to maximize $\nrel(G,p)$. This all hinges on how likely it is that nodes will fail and this may change considerably from the time the network is designed as the components age. Therefore, it is of interest to find graphs $G$ such that $\nrel(G,p)\ge \nrel(H,p)$ for all values of $p\in [0,1]$ and all graphs $H$ with the same order and number of edges as $G$ (i.e. can be constructed with the same resources). Such a graph is called \textit{uniformly most reliable} (UMR) with respect to node reliability. For graphs on $n$ nodes and $m$ edges, UMR graphs with respect to node reliability do not always exist, as was shown for $m=n$ \cite{StivarosThesis}, but there are some cases where UMR graphs do exist. A complete list of which is $K_{1,n-1}$ and $K_n-M$ for a matching $M$ \cite{StivarosThesis}, $K_{b,b+2}$ \cite{goldschmidt1994reliability}, and $K_{b,b+1,b+1,\ldots,b+1,b+2}$ \cite{YuShaoMeng2010}. The last family is a generalization of $K_{b,b+1,b+2}$, first shown in \cite{completetripartite}. As is readily noted, all known UMR graphs are complete multipartite. As such, there was an appetite to consider variations on node reliability to get a more diverse sample of UMR graphs when nodes can fail. 

Let $G$ be a graph where each node is operational independently with probability $p$ and $t,s\in V(G)$ be target nodes . Then the \textit{two-terminal node reliability} of $G$ is the probability that the graph is connected and contains $t$ and $s$. In a very recent paper \cite{twoterminalnodereliability}, it was shown that for every $n$ and $m$ there exists a UMR graph with respect to two-terminal node reliability, however when the target vertices are required to have distance at least $3$ from each other, there exists a UMR graph on $n$ vertices and $m$ edges if and only if $m\le 8$, $m\ge \left\lfloor\frac{(n-1)^2}{4}\right\rfloor$, or $n=9$ and $m=12$. This complete classification is an important step for node reliability, but two target vertices are not always clearly defined in the network the graph is modelling. Thus, we introduce a new variant on node reliability and demonstrate that it also gives UMR graphs that are not complete multipartite, and in particular, sparse ones. Our model in a nutshell is to find the probability that the graph is not only connected, but also that the well-studied game of Cops and Robber can be won by a single cop. 

The game of \textit{Cops and Robber} on graphs introduced in \cite{NowakowskiWinkler1983} and independently in \cite{Quilliot} is a two player pursuit-evasion game where one player, a set of $k$-cops, attempts to capture the other, a robber. The game begins by placing each cop on a vertex of the graph. The robber is then placed on a vertex. The players then take turns moving to an adjacent vertex or remaining on their current vertex. The cops win if one can eventually occupy the same vertex as the robber and the robber wins if they can evade capture indefinitely. A graph $G$ is \textit{cop-win} if a single cop can always guarantee a win playing on $G$. Cops and Robber has been studied extensively since its first introduction, see \cite{BonatoNowakowski2011} for an overview and \cite{ScottAudakov2011,BonatoMohar2020,intersection,BollobasKunLeader2013,HosseiniKnoxMohar2020} for some recent work on the game. 

For us, the game serves as a stronger notion of reliability. Since node failures may leave the network more vulnerable to attack, the graph being cop-win serves as a proxy for a network where a malicious agent can be efficiently neutralized after node failures. As with node reliability, let $G$ be a graph whose nodes are operational independently with probability $p$, for some $p\in[0,1]$ and whose edges are operational if and only if both of their endpoints are operational. We introduce the \textit{node cop-win reliability} of a graph to be defined as the probability that the operational vertices induce a cop-win graph. We let $\ncrel(G,p)$ denote the node cop-win reliability of $G$. Let $W_i(G)$ denote the number of induced cop-win subgraphs of $G$ of order $i$. Then we can express $\ncrel(G,p)$ as the following:
$$
\ncrel(G,p)=\sum\limits_{i=1}^n W_i(G)(1-p)^{n-i}p^i.
$$

Constructing graphs that are UMR with respect to node cop-win reliability is the focus of this paper which will be laid out as follows. In Section~\ref{sec:preliminaries} we present some preliminary results and definitions. These will be used in Section~\ref{sec:unicyclic} to show that, in contrast to node reliability, there is a UMR graph with respect to node cop-win reliability among all unicyclic graphs of order $n\ge 5$. This is despite the fact that the UMR graph does not maximize $W_i(G)$ for all $i$ which is of note as all known UMR graphs with respect to node reliability were shown to also maximize $S_i(G)$ for all $i$. Therefore, we introduce new methods for showing graphs are UMR involving the analytic properties of related polynomials. In Section~\ref{sec:bicylic}, we show that there is no UMR bicyclic graph with respect to node reliability, but that there is one for order $n\ge 7$ with respect to node cop-win reliability. Our proofs in this case make extensive use of the notion of ``pivoting'' introduced by Stivaros \cite{StivarosThesis} for computing the node reliability of a graph recursively. Pivoting has, for whatever reason, not been used to show any graphs are UMR since Stivaros did in 1990. Given the power of this method for inductive proofs of UMR graphs, we expect our methods will inspire further work in this area using pivoting techniques and our new analytic techniques. Finally, we conclude with some open problems for future research.\\

\noindent\textbf{Note on Terminology:} We note that when node reliability was originally introduced by Stivaros \cite{StivarosThesis}, it was referred to as the residual connectedness reliability of a network and uniformly most reliable graphs were called uniformly best.

\section{Preliminaries}\label{sec:preliminaries}
In this work all graphs considered are simple, finite, and undirected. The number of nodes (vertices) of the graph is the \textit{order} of the graph. The \textit{disjoint union} of two graphs $G$ and $H$, denoted $G\cup H$ is defined as the graph $(V(G)\cup V(H),E(G)\cup E(H))$. The disjoint union of $k$ copies of a graph $G$ is denoted $kG$. The \textit{join} of $G$ and $H$, denoted $G\vee H$, is obtained from $G\cup H$ by adding the edges $uv$ for all $u\in V(G)$ and all $v\in V(H)$. For a graph $G$ with $v\in V(G)$, $G-v$ denotes the graph obtained by deleting $v$ and all of its incident edges. The \textit{closed neighbourhood} of $v$, denoted $N[v]$ is the set $\{u\in V(G):uv\in E(G)\}\cup\{v\}$. The graph $G/v$ is obtained from $G-v$ by adding in all missing edges between any pair of vertices in $N[v]$. A \textit{unicyclic} graph is a connected graph with $|V(G)|=|E(G)|$. A \textit{bicyclic} graph is a connected graph with $|V(G)|+1=|E(G)|$.
For two polynomials $A(x)=\sum_{k=0}^na_kx^k$ and $B(x)=\sum_{k=0}^mb_kx^k$, we write $B(x)\preceq A(x)$ if $m\le n$ and $b_k\le a_k$ for all $0\le k\le n$, i.e. $A(x)$ is coefficient-wise greater than $B(x)$. 

While there is an elegant structural characterization for cop-win graphs given in \cite{NowakowskiWinkler1983}, for our purposes it will be sufficient to use the result from \cite{anstee1988bridged} that connected chordal graphs are cop-win. A graph is \textit{chordal} if it contains no induced cycles of order greater than $3$. There are graphs that are cop-win and not chordal (wheels for example) but it is straightforward to see that $C_n$ is not cop-win for all $n\ge 4$. Chordal graphs are important for our study of node cop-win reliability because of the following result.

\begin{lemma}\label{lem:polysequaliffchordal}
$\nrel(G,p) = \ncrel(G,p)$ if and only if $G$ is a chordal graph.
\end{lemma}
\begin{proof}
Since all cop-win graphs are connected, it is clear that $S_i(G)\geq W_i(G)$ for all $i$. If a graph $G$ contains an induced cycle of order $j$, where $j\geq 4$, then we know that $S_j(G)>W_j(G)$ and thus the coefficients $S_j(G)$ and $W_j(G)$ will not be equal. Therefore, if $G$ contains any induced cycles of order $\geq 4$, then $\nrel(G,p)\neq\ncrel(G,p)$. Furthermore, every induced connected subgraph of a chordal graph is also a connected chordal graph, and every connected chordal graph is cop-win~\cite{anstee1988bridged}, therefore, $S_i(G)=W_i(G)$ for all $i$ and therefore, $\nrel(G,p) = \ncrel(G,p)$.       
\end{proof}


For our work on $\nrel(G,p)$ and $\ncrel(G,p)$ it will be convenient to define two more closely related polynomials. The first was introduced in \cite{BrownMol2016} and is called the \textit{connected set polynomial} of a graph, denoted $\cs(G,x)$, defined by $$\cs(G,x)=\sum_{k=1}^{n}S_k(G)x^k.$$
The second is a new polynomial which we call the \textit{cop-win polynomial} of a graph, denoted $\cw(G,x)$, defined by $$\cw(G,x)=\sum_{k=1}^{n}W_k(G)x^k.$$

For our work on unicyclic and bicyclic graphs, there are five graphs whose cop-win and connected set polynomials need to be explicitly computed for our proofs. One is $C_n$, the cycle of order $n$, and the remaining are defined below and depicted in Figure~\ref{fig:importantgraphs}.
\begin{itemize}
\item Let $U_n$ be the graph $(K_2\cup\overline{K_{n-3}})\vee K_1$.
\item Let $A_n$ be the cycle graph of order $n-1$ with a single leaf adjacent to one vertex of the cycle.
\item Let $B_n$ be the graph obtained from $U_{n-1}$ by joining a new vertex to the universal vertex and to one of the vertices of degree $2$.
\item Let $F(n-5,0)$ be the graph obtained from the star on $n$ vertices by joining two disjoint pairs of leaves with edges. Note that the notation will be explained in Section~\ref{sec:bicylic} when this graph is generalized with another parameter.
\end{itemize} 

\setcounter{subfigure}{0}
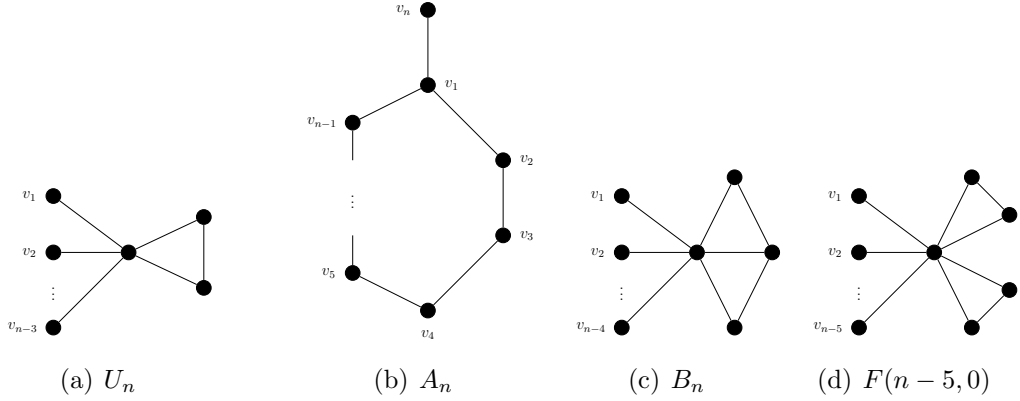
\begin{figure}[!h]
\def\c{0.5}
\def\r{2}
\centering
\subfigure[$U_n$]{
\scalebox{\c}{
\begin{tikzpicture}
\begin{scope}[every node/.style={circle,thick,draw,fill}]

     \node[label=left:$v_{n-3}$] (13) at (-1*\r,-1*\r) {}; 
     \node[label=left:$v_2$](11) at (-1*\r,0*\r) {};   
    \node[label=left:$v_1$](6) at (-1*\r,0.75*\r) {};
    
    \node(5) at (0*\r,0*\r) {};
\node(4) at (1*\r,sqrt{2/3}*\r) {};
    \node (3) at (1*\r,-sqrt{2/3}*\r) {}; 
      
\end{scope}

\begin{scope}
   
    \path [-] (5) edge node {} (4);
    
    \path [-] (4) edge node {} (3);
    \path [-] (5) edge node {} (3);
    
    \path [-] (6) edge node {} (5);
    \path [-] (5) edge node {} (11);
    \path [-] (5) edge node {} (13);
\end{scope}

\path (11) -- node[auto=false]{\vdots} (13);
\end{tikzpicture}}}
\qquad
\subfigure[$A_n$]{
\scalebox{\c}{
\begin{tikzpicture}

\begin{scope}[every node/.style={circle,thick,fill,draw}]
    \node[label=left:$v_{n}$] (1) at (0*\r,1*\r) {};
    \node[label=right:$v_{1}$] (2) at (0*\r,0*\r) {};
    \node[label=right:$v_{2}$] (3) at (1*\r,-1*\r) {};
    \node[label=right:$v_{3}$] (4) at (1*\r,-2*\r) {};
    \node[label=below:$v_{4}$] (5) at (0*\r,-3*\r) {};
    \node[label=left:$v_{5}$] (6) at (-1*\r,-2.5*\r) {};
    \node[label=left:$v_{n-1}$] (7) at (-1*\r,-0.5*\r) {};
\end{scope}

\begin{scope}

    \path [-] (1) edge node {} (2);
    \path [-] (2) edge node {} (3);
    \path [-] (3) edge node {} (4);
    \path [-] (5) edge node {} (4);
    \path [-] (5) edge node {} (6);
    \path [-] (7) edge node {} (2);
    
    \path [-] (7) edge node {} (-1*\r,-1*\r);
    \path [-] (6) edge node {} (-1*\r,-2*\r);
\end{scope}

\path (6) -- node[auto=false]{\vdots} (7);

\end{tikzpicture}}}
\subfigure[$B_n$]{
\scalebox{\c}{
\begin{tikzpicture}
\begin{scope}[every node/.style={circle,thick,draw,fill}]

     \node[label=left:$v_{n-4}$] (13) at (-1*\r,-1*\r) {}; 
     \node[label=left:$v_2$](11) at (-1*\r,0*\r) {};   
     \node[label=left:$v_1$](6) at (-1*\r,0.75*\r) {};
    
     \node(5) at (0*\r,0*\r) {};
     \node(4) at (1*\r,0*\r) {};
     \node (3) at (0.5*\r,-1*\r) {};
     \node(1) at (0.5*\r,1*\r) {};

\end{scope}

\begin{scope}
   
    \path [-] (5) edge node {} (4);
    
    \path [-] (4) edge node {} (3);
    \path [-] (5) edge node {} (3);
    
    \path [-] (6) edge node {} (5);
    \path [-] (5) edge node {} (11);
    \path [-] (5) edge node {} (13);
    
	\path [-] (1) edge node {} (4);
    \path [-] (1) edge node {} (5);    
    
\end{scope}

\path (11) -- node[auto=false]{\vdots} (13);
\end{tikzpicture}}}
\subfigure[$F(n-5,0)$]{
\scalebox{\c}{
\begin{tikzpicture}
\begin{scope}[every node/.style={circle,thick,draw,fill}]

     \node[label=left:$v_{n-5}$] (13) at (-1*\r,-1*\r) {}; 
     \node[label=left:$v_2$](11) at (-1*\r,0*\r) {};   
     \node[label=left:$v_1$](6) at (-1*\r,0.75*\r) {};
    
     \node(5) at (0*\r,0*\r) {};
     \node(4) at (1*\r,0.5*\r) {};
     \node (3) at (0.5*\r,-1*\r) {};
     \node(1) at (0.5*\r,1*\r) {};
     \node(0) at (1*\r,-0.5*\r) {};

\end{scope}

\begin{scope}
   
    \path [-] (5) edge node {} (4);
    
    \path [-] (1) edge node {} (5);
    \path [-] (5) edge node {} (3);
    
    \path [-] (6) edge node {} (5);
    \path [-] (5) edge node {} (11);
    \path [-] (5) edge node {} (13);
    
	\path [-] (1) edge node {} (4);
    \path [-] (0) edge node {} (5);   
    \path [-] (0) edge node {} (3); 
    
\end{scope}

\path (11) -- node[auto=false]{\vdots} (13);
\end{tikzpicture}}}
\caption{Important unicyclic and bicyclic graphs.}\label{fig:importantgraphs}
\end{figure}

\begin{lemma}\label{lem:explictformulasUnAnCn}
For $n\ge 5$,

\begin{enumerate}[nosep]
\item[1)] $\cw(U_n,x)=\cs(U_n,x)=nx+nx^2+\sum\limits_{k=3}^{n}\binom{n-1}{k-1}x^k$.
\item[2)] $\cw(C_n,x)=\cs(C_n,x)-x^n=\sum\limits_{k=1}^{n-1}nx^k$.
\item[3)] $\cw(A_n,x)=\cs(C_n,x)-x^n-x^{n-1}=nx+\sum\limits_{k=2}^{n-2}(n+k-2)x^k+(n-1)x^{n-1}$.
\item[4)] $\cw(B_n,x)=\cs(B_n,x)=nx+(n+1)x^2+\left(\binom{n-1}{2}+1\right)x^3+\sum\limits_{k=4}^{n}\binom{n-1}{k-1}x^k$.
\item[5)] $\cw(F(n-5,0),x)=\cs(F(n-5,0),x)=nx+(n+1)x^2+\sum\limits_{k=3}^{n}\binom{n-1}{k-1}x^k$.
\end{enumerate}

\end{lemma}

\begin{proof}
1) $U_n$ is a chordal graph, so by Lemma~\ref{lem:polysequaliffchordal}, $\cw(U_n,x)=\cs(U_n,x)$. Now every subset of $3$ vertices that does not contain the universal vertex is isomorphic to $K_2\cup 2K_1$ or $3K_1$ and therefore disconnected. Further, every subset of $3$ vertices that contains the universal vertex is clearly connected. Hence, for $k>2$, $S_k(U_n)=\binom{n-1}{k-1}$.  

2) Since any proper connected induced subgraph of $C_n$ will be chordal, the only connected induced subgraph of $C_n$ that is not cop-win is the graph itself. Therefore, $\cw(C_n,x)=\cs(C_n,x)-x^n$ and $\cs(C_n,x)=\sum\limits_{k=1}^{n-1}nx^k+x^n$ from the proof of Theorem 3.3 in \cite{BrownMol2016}.

3) For $3\le i \le n-2$, $S_i(A_n)$ is equal to $S_i(C_{n-1})=n-1$ plus the number of connected sets of order $i$ in $A_n$ that contain the leaf. These sets are in one-to-one-correspondence with the connected sets of order $i-1$ in $C_{n-1}$ that contain the leaf's neighbour. In turn, these connected sets are in one-to-one correspondence with the number of subpaths of order $n-2-(i-1)$ of $P_{n-2}$. Now,  from an identical argument to Lemma 2.3 in \cite{CameronMol2020}, we have $\cs(A_n,x)=nx+\sum\limits_{k=2}^{n-2}(n+k-2)x^k+nx^{n-1}+x^n$. Finally, since $n\ge 5$, the only connected induced subgraphs of $A_n$ that are not chordal are induced by the entire vertex set or by the entire vertex set of $C_{n-1}$. Hence, $\cw(A_n,x)=\cs(A_n,x)-x^{n-1}-x^n$.

4) and 5) are proved in a very similar manner to 1), so we omit the proofs in these cases.
\end{proof}

We can now introduce the key notion of ``pivoting'' that we will apply extensively throughout our work. The power of this result is that it allows for $\nrel(G,p)$ to be computed recursively and therefore opens up the possibility of inductive proofs that graphs are UMR. 

\begin{theorem}[Pivoting Theorem \cite{StivarosThesis}]\label{thm:nrelpivot}
Let $G$ be a graph of order $n$ and $v\in V(G)$. Then
\begin{align*}
\nrel(G,p)=&(1-p)\nrel(G-v,p)+p\nrel(G/v,p)\\
&-p\left((1-p)^{\deg(v)}\nrel(G-N[v],p)+(1-p)^{n-1} \right).
\end{align*}
\end{theorem}

A straightforward corollary of this (either using the same counting argument or the transformation from $\nrel(G,p)$ to $\cs(G,x)$) is the following.
\begin{corollary}\label{cor:cspivot}
Let $G$ be a graph of order $n$ and $v\in V(G)$. Then
$$\cs(G,x)=\cs(G-v,x)+x\left(\cs(G/v,x)-\cs(G-N[v],x)+1 \right). $$
\end{corollary}

These results allow for the following lemma that will be applied extensively throughout our work on unicyclic and bicyclic graphs.

\begin{lemma}\label{lem:pivotcomps}
Let $G$ and $H$ be two graphs such that there exists $v\in V(G)$ and $u\in V(H)$ such that 
\begin{itemize}
\item[1)] $\cs(G-v,x)\preceq \cs(H-u,x)$,
\item[2)] $\cs(G/v,x)\preceq \cs(H/u,x)$, and
\item[3)] $\cs(H-N[u],x)\preceq \cs(G-N[v],x)$.
\end{itemize}
Then $\cs(G,x)\preceq \cs(H,x)$.
\end{lemma}
\begin{proof}
Let $G$, $H$, $u$, and $v$ be as in the hypothesis. Then by Corollary~\ref{cor:cspivot},
\begin{align*}
\cs(H,x)-\cs(G,x)&=\cs(H-u,x)+x\left(\cs(H/u,x)-\cs(H-N[u],x)+1 \right)\\
&-\left(\cs(G-v,x)+x\left(\cs(G/v,x)-\cs(G-N[v],x)+1 \right)\right)\\
&=\cs(H-u,x)-\cs(G-v,x)\\
&+x\left(\cs(H/u,x)-\cs(G/v,x) \right)\\
&+x\left(\cs(G-N[v],x)- \cs(H-N[u],x)\right).
\end{align*}
By assumption, each of $\cs(H-u,x)-\cs(G-v,x)$, $x\left(\cs(H/u,x)-\cs(G/v,x) \right)$, and $x\left(\cs(G-N[v],x)- \cs(H-N[u],x)\right)$ has all nonnegative coefficients, so their sum does as well. Therefore, $\cs(G,x)\preceq \cs(H,x)$.
\end{proof}

\section{Unicyclic Graphs}\label{sec:unicyclic}
We are now ready to explore UMR unicyclic graphs with respect to node cop-win reliability. We begin with a lemma that will be helpful for some special cases.

\begin{lemma}\label{lem:AppleCoeffsSmaller}
For $n\ge 5$ and $3\le k \le n-2$, $$\binom{n-1}{k-1}\ge n+k-2.$$
\end{lemma}
\begin{proof}
It is well-known that the sequence $\binom{n-1}{1},\binom{n-2}{2},\ldots,\binom{n-1}{n-1}$ is unimodal with mode at $\lfloor(n-1)/2 \rfloor$ and $\binom{n-1}{k}=\binom{n-1}{n-1-k}$. Therefore, for 
$3\le k\le n-2$,
\begin{align*}
\binom{n-1}{k-1}&\ge \binom{n-1}{n-3}\\
&=\frac{(n-1)(n-2)}{2}\\
&\ge 2n+4\ \ \ \ \ \ \ \ \ \ \text{(since $n^2-7n+10\ge 0$ for all $n\ge 5$)}\\
&\ge n+k-2.
\end{align*}
\end{proof}

It is noted in \cite{StivarosThesis} that there is no UMR unicyclic graph with respect to node reliability because $\nrel(U_n,p)$ will be largest for values of $p$ close to $0$ and $\nrel(C_n,p)$ will be largest for values of $p$ close to $1$. For node cop-win reliability this issue is avoided, however we cannot prove that $\cw(C_n,x)\preceq\cw(U_n,x)$ to show that $\ncrel(C_n,p)\le\ncrel(U_n,p)$ for all $p\in [0,1]$ as $W_{n-1}(C_n)>W_{n-1}(U_n)$. Therefore, we require a new analytic technique for showing one graph is more reliable than another.

\begin{lemma}\label{lem:UnbiggerthanCn}
For $n\ge 5$, $\ncrel(U_n,p)>\ncrel(C_n,p)$ for all $p\in (0,1]$.
\end{lemma}
\begin{proof}
Since $\ncrel(U_n,1)=1>0=\ncrel(C_n,1)$, it suffices to show that $\ncrel(U_n,p)-\ncrel(C_n,p)$ has no roots in the interval $(0,1]$ to show that $\ncrel(U_n,p)>\ncrel(C_n,p)$ for all $p\in (0,1)$. We will look at the cop-win polynomial for our analysis. Note that 
$$\cw(U_n,p)-\cw(C_n,p)=\left(1+p\right)^{n}\left(\ncrel\left(U_n,\frac{p}{1+p}\right)-\ncrel\left(C_n,\frac{p}{1+p}\right) \right).$$

Since $f(p)=\frac{p}{1+p}$ is a M\"{o}bius transformation, it is a continuous bijection from $\mathbb{C}\cup\{\infty\}$ to $\mathbb{C}\cup\{\infty\}$ (see \cite{BrownMol2016,BrownCameron2018} for use of M\"{o}bius transformations for roots of graph polynomials and \cite{Fisher} for further details on M\"{o}bius transformations in general). It is also clear that $p\in\mathbb{R}\cup \{\infty\}$ if and only if $\frac{p}{1+p}\in\mathbb{R}\cup\{\infty\}$. Since $f(p)$ is a continuous bijection, and $f(0)=0$ and $f(\infty)=1$, it follows that  $\frac{p}{1+p}\in (0,1)$ if and only if $p>0$. Finally, for $p\neq -1$, $\frac{p}{1+p}$ is a root of $\ncrel\left(U_n,\frac{p}{1+p}\right)-\ncrel\left(C_n,\frac{p}{1+p}\right)$ if and only if $p$ is a root of $\cw(U_n,p)-\cw(C_n,p)$. Therefore, it suffices to show that $\cw(U_n,p)-\cw(C_n,p)$ has no positive real roots. 

From Lemma~\ref{lem:explictformulasUnAnCn}, we know that, 
\begin{align}
\cw(U_n,x)-\cw(C_n,x)&=x^n+\sum_{k=3}^{n-1}\left(\binom{n-1}{k-1}-n \right)x^k.\label{eq:cwdiffformulaUnCn}
\end{align}

It follows from Lemma~\ref{lem:AppleCoeffsSmaller} that the coefficient of $x^k$ in $\cw(U_n,x)-\cw(C_n,x)$ is positive for $k=3,4,\ldots,n-2$ and is equal to $-1$ for $k=n-1$. Therefore, from (\ref{eq:cwdiffformulaUnCn}), if $x\in (0,1)$, then
\begin{align}
\cw(U_n,x)-\cw(C_n,x)&>x^n+\left(\binom{n-1}{n-3}-n-1\right)x^{n-2}+  \sum_{k=3}^{n-3}\left(\binom{n-1}{k-1}-n \right)x^k.\label{ineq:xin01}
\end{align}

Now, every coefficient in (\ref{ineq:xin01}) is positive, so it has no positive roots, and therefore, $\cw(U_n,x)-\cw(C_n,x)>0$ for all $x\in(0,1)$. Similarly, for all $x\in [1,\infty)$,
$$\cw(U_n,x)-\cw(C_n,x)\ge \sum_{k=3}^{n-2}\left(\binom{n-1}{k-1}-n \right)x^k>0.$$

Therefore, $\cw(U_n,x)-\cw(C_n,x)$ has no roots in $(0,\infty)$, so $\ncrel(U_n,p)-\ncrel(C_n,p)$ has no roots in $(0,1)$. Hence, $\ncrel(U_n,p)>\ncrel(C_n,p)$ for all $p\in (0,1]$.

\end{proof}

We will require one more lemma before we can show that $U_n$ is UMR with respect to node cop-win reliability.

\begin{lemma}\label{lem:twoedges}
Let $G$ be a unicyclic graph of order $n$ such that $G\not\cong U_n$ and $G\not\cong C_n$ and $v\in V(G)$ be a leaf. Then $G-N[v]$ has at least two edges.
\end{lemma}
\begin{proof}
Suppose by way of contradiction that $G$ is a unicyclic graph of order $n$ other than $U_n$ and $C_n$ such that $G-N[v]$ has at most one edge for some leaf $v\in V(G)$. Let $u\in V(G)$ be the neighbour of $v$. Since $v$ is a leaf, $v$ is incident to exactly one edge. Since $G$ has $n$ edges and $G-N[v]$ has at most one edge, it follows that $u$ has at least $n-2$ neighbours other than $v$, so that $\deg(u)\ge n-1$. However, the only unicyclic graph with a universal vertex is $U_n$, a contradiction.
\end{proof}

\begin{theorem}\label{thm:UnbiggerthanallexcpetCn}
Let $G$ be a unicyclic graph of order $n\ge 5$ not isomorphic to $C_n$. Then $\cs(G,x)\preceq\cs(U_n,x)$.
\end{theorem}
\begin{proof}
If $G=A_n$, then the result follows immediately from Lemma~\ref{lem:explictformulasUnAnCn} and Lemma~\ref{lem:AppleCoeffsSmaller}. So from now on, suppose $G$ is not $A_n$. The proof proceeds by induction on $n$. For $n=5$, the result can readily be verified by comparing the two unicyclic graphs of order $5$ other than $C_n$ and $A_n$ with $U_n$. Now suppose the result holds for unicyclic graphs of order $n-1$ for some $5\le n-1$. Let $G$ be a unicyclic graph of order $n$ that is not $C_n$ nor $A_n$. Let $v$ be a leaf of $G$ (noting that one exists since $G$ is not $C_n$). Now, $G-v=G/v$ is a unicyclic graph of order $n-1$ and is not equal to $C_{n-1}$ since $G$ is not $A_{n}$. Let $u$ be a leaf of $U_{n}$. Note that $U_{n}-u=U_{n}/u=U_{n-1}$.
Thus, $\cs(G-v,x)\preceq \cs(U_n-u,x)$ and $\cs(G/v,x)\preceq \cs(U_n/u,x)$ by the inductive hypothesis. Further, $\cs(G-N[v],x)=(n-2)x+2x^2+f(x)$ where $f(x)$ is a polynomial with positive coefficients since $G-N[v]$ has at least two edges by Lemma~\ref{lem:twoedges}. Since $U_{n}-N[u]=K_2\cup \overline{K_{n-4}}$, it follows that $\cs(U_{n+1}-N[u],x)\preceq \cs(G-N[v],x)$. Hence, by Lemma~\ref{lem:pivotcomps}, $\cs(G,x)\preceq \cs(U_{n+1},x)$.
\end{proof}

From the previous result, it follows that although $U_n$ is not UMR with respect to node reliability, it is more reliable than every unicyclic graph other than $C_n$. We now have all necessary results to prove the main theorem of this section.

\begin{theorem}
For $n\ge 5$, $U_n$ is UMR with respect to node cop-win reliability among all unicyclic graphs of order $n$.
\end{theorem}
\begin{proof}
Let $n\ge 5$ and $G$ be a unicyclic graph of order $n$. 

If $G\cong C_n$, then $\ncrel(G,p)\le \ncrel(U_n,p)$ for all $p\in [0,1]$ by Lemma~\ref{lem:UnbiggerthanCn}.

If $G\not\cong C_n$, then, by Theorem~\ref{thm:UnbiggerthanallexcpetCn},
$$\ncrel(G,p)\le\nrel(G,p)\le\nrel(U_n,p)=\ncrel(U_n,p)$$
for all $p\in [0,1]$.  
\end{proof}

\section{Bicyclic Graphs}\label{sec:bicylic}
We now turn our attention to bicyclic graphs although our results on unicyclic graphs from the previous section will be applied. We begin by noting a classification of bicyclic graphs based on their structure. A depiction of each graph introduced below can be found in Figure~\ref{fig:bicyclic forms}.

\begin{itemize}
\item Let $G_1(a,b)$ be the bicyclic graph obtained by identifying a vertex of $C_a$ with a vertex of $C_b$.
\item  Let $G_2(a,b,c)$ be the bicyclic graph obtained by identifying one leaf of $P_c$ with a vertex of $C_a$ and identifying the other leaf of $P_c$ with a vertex of $C_b$. Note that $c\ge 2$.
\item Let $G_3(a,b,c)$ be the graph consisting of two given vertices joined by three disjoint
paths whose orders are $a, b$, and $c$, respectively, where $a, b, c \ge 0$ and at most one of them is 0.
\end{itemize}

\setcounter{subfigure}{0}
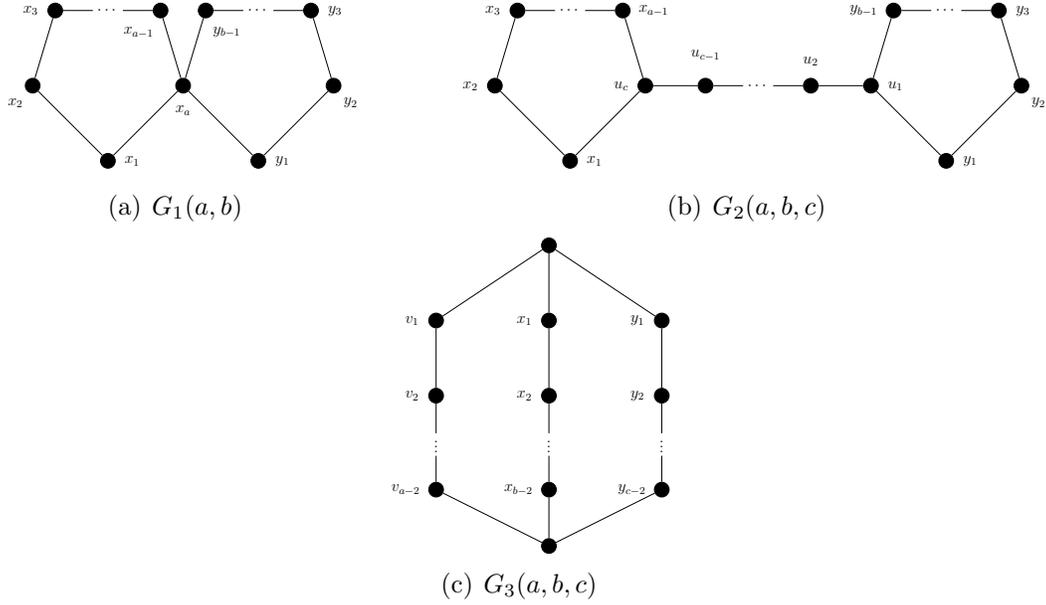
\begin{figure}[!h]
\def\c{0.5}
\def\r{2}
\centering
\subfigure[$G_1(a,b)$]{
\scalebox{\c}{
%
%
%
%
%
%
%
%
%
%
%
\begin{tikzpicture}

\begin{scope}[every node/.style={circle,thick,fill,draw}]
	\node[label=right:$x_{1}$] (2) at (0*\r,0*\r) {};       
    \node[label=below left:$x_{2}$] (1) at (-1*\r,1*\r) {};
	\node[label=below:$x_{a}$] (5) at (1*\r,1*\r) {};    
    \node[label=left:$x_{3}$] (3) at (-0.7*\r,2*\r) {};
	\node[label=below left:$x_{a-1}$] (4) at (0.7*\r,2*\r) {};  
	
	\node[label=right:$y_{1}$] (7) at (2*\r,0*\r) {};       
	\node[label=below right:$y_{2}$] (10) at (3*\r,1*\r) {};    
    \node[label=below right:$y_{b-1}$] (8) at (1.3*\r,2*\r) {};
	\node[label=right:$y_{3}$] (9) at (2.7*\r,2*\r) {};

\end{scope}

\begin{scope}

    \path [-] (5) edge node {} (2);
    \path [-] (5) edge node {} (4);
    \path [-] (1) edge node {} (3);
    \path [-] (2) edge node {} (1);
    \path [-] (3) edge node {} (-0.2*\r,2*\r);
    \path [-] (4) edge node {} (0.2*\r,2*\r);
    
    \path [-] (5) edge node {} (8);
    \path [-] (5) edge node {} (7);

    \path [-] (10) edge node {} (7);
    \path [-] (10) edge node {} (9);
    \path [-] (8) edge node {} (1.8*\r,2*\r);
    \path [-] (9) edge node {} (2.2*\r,2*\r);
    
\end{scope}

\path (3) -- node[auto=false]{$\cdots$} (4);
\path (8) -- node[auto=false]{$\cdots$} (9);

\end{tikzpicture}}}
\qquad
\subfigure[$G_2(a,b,c)$]{
\scalebox{\c}{
\begin{tikzpicture}

\begin{scope}[every node/.style={circle,thick,fill,draw}]
	\node[label=right:$x_{1}$] (2) at (0*\r,0*\r) {};       
    \node[label=left:$x_{2}$] (1) at (-1*\r,1*\r) {};
	\node[label=left:$u_{c}$] (5) at (1*\r,1*\r) {};    
    \node[label=left:$x_{3}$] (3) at (-0.7*\r,2*\r) {};
	\node[label=right:$x_{a-1}$] (4) at (0.7*\r,2*\r) {};  
	
	\node[label=right:$y_{1}$] (7) at (5*\r,0*\r) {};       
    \node[label=right:$u_{1}$] (6) at (4*\r,1*\r) {};
	\node[label=below right:$y_{2}$] (10) at (6*\r,1*\r) {};    
    \node[label=left:$y_{b-1}$] (8) at (4.3*\r,2*\r) {};
	\node[label=right:$y_{3}$] (9) at (5.7*\r,2*\r) {}; 
	
	\node[label=above:$u_{2}$] (11) at (3.2*\r,1*\r) {};       
    \node[label=above:$u_{c-1}$] (12) at (1.8*\r,1*\r) {};

\end{scope}

\begin{scope}

    \path [-] (5) edge node {} (2);
    \path [-] (5) edge node {} (4);
    \path [-] (1) edge node {} (3);
    \path [-] (2) edge node {} (1);
    \path [-] (3) edge node {} (-0.2*\r,2*\r);
    \path [-] (4) edge node {} (0.2*\r,2*\r);
    
    \path [-] (10) edge node {} (7);
    \path [-] (10) edge node {} (9);
    \path [-] (6) edge node {} (8);
    \path [-] (7) edge node {} (6);
    \path [-] (8) edge node {} (4.8*\r,2*\r);
    \path [-] (9) edge node {} (5.2*\r,2*\r);
    
	\path [-] (11) edge node {} (6);
	\path [-] (12) edge node {} (5);    
    \path [-] (11) edge node {} (2.7*\r,1*\r);
    \path [-] (12) edge node {} (2.3*\r,1*\r);
\end{scope}

\path (3) -- node[auto=false]{$\cdots$} (4);
\path (8) -- node[auto=false]{$\cdots$} (9);
\path (11) -- node[auto=false]{$\cdots$} (12);

\end{tikzpicture}}}
\subfigure[$G_3(a,b,c)$]{
\scalebox{\c}{
\begin{tikzpicture}

\begin{scope}[every node/.style={circle,thick,fill,draw}]
    \node (i) at (0*\r,4*\r) {};
    \node (j) at (0*\r,0*\r) {};
    \node[label=left:$v_{1}$] (v1) at (-1.5*\r,3*\r) {};
	\node[label=left:$v_{2}$] (v2) at (-1.5*\r,2*\r) {};   
	\node[label=left:$v_{a-2}$] (vp) at (-1.5*\r,0.75*\r) {};
    \node[label=left:$x_{1}$] (x1) at (0*\r,3*\r) {};
	\node[label=left:$x_{2}$] (x2) at (0*\r,2*\r) {};   
	\node[label=left:$x_{b-2}$] (xq) at (0*\r,0.75*\r) {};
    \node[label=left:$y_{1}$] (y1) at (1.5*\r,3*\r) {};
	\node[label=left:$y_{2}$] (y2) at (1.5*\r,2*\r) {};   
	\node[label=left:$y_{c-2}$] (yr) at (1.5*\r,0.75*\r) {};
    
\end{scope}


    \path [-] (i) edge node {} (v1);
    \path [-] (i) edge node {} (x1);
    \path [-] (i) edge node {} (y1);
    \path [-] (j) edge node {} (vp);
    \path [-] (j) edge node {} (xq);
    \path [-] (j) edge node {} (yr);
    
    \path [-] (x1) edge node {} (x2);
    \path [-] (y1) edge node {} (y2);
    \path [-] (v1) edge node {} (v2);
    
    \path [-] (vp) edge node {} (-1.5*\r,1.2*\r);
    \path [-] (xq) edge node {} (0*\r,1.2*\r);
    \path [-] (yr) edge node {} (1.5*\r,1.2*\r);
    
    \path [-] (v2) edge node {} (-1.5*\r,1.5*\r);
    \path [-] (x2) edge node {} (0*\r,1.5*\r);
    \path [-] (y2) edge node {} (1.5*\r,1.5*\r);

\path (v2) -- node[auto=false]{\vdots} (vp);
\path (x2) -- node[auto=false]{\vdots} (xq);
\path (y2) -- node[auto=false]{\vdots} (yr);

\end{tikzpicture}}}
\caption{The bicyclic graphs $G_1(a,b)$, $G_2(a,b,c)$, and $G_3(a,b,c)$}\label{fig:bicyclic forms}
\end{figure}

   We will call bicyclic graphs with $G_1(a,b)$ as an induced subgraph, \textit{Type 1} bicyclic graphs; $G_2(a,b, c)$ as an induced subgraph, \textit{Type 2} bicyclic graphs; and $G_3(a,b,c)$ as an induced subgraph, \textit{Type 3} bicyclic graphs. Note that every bicyclic graph is of exactly one of these types. 
   
\subsection{Type 1 and Type 2}
An important exceptional Type 1 bicyclic graph that generalizes the graph $F(n-5,0)$ will need to be handled separately. Define the graph $F(n_1,n_2)$ as being obtained from $F(n-5,0)$ by subdividing $n_2$ of the $n-5$ edges incident with leaves (see Figure~\ref{fig:F(n1,n2)}).

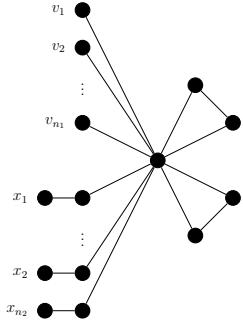
\begin{figure}
\def\c{0.5}
\def\r{2}
\begin{center}
\scalebox{\c}{
\begin{tikzpicture}
\begin{scope}[every node/.style={circle,thick,draw,fill}]

     \node[label=left:$v_{n_1}$] (13) at (-1*\r,0.5*\r) {}; 
     \node[label=left:$v_2$](11) at (-1*\r,1.5*\r) {};   
     \node[label=left:$v_1$](6) at (-1*\r,2*\r) {};
     
     \node[label=left:$x_{n_2}$] (33) at (-1.5*\r,-2*\r) {}; 
     \node[label=left:$x_2$](32) at (-1.5*\r,-1.5*\r) {};   
     \node[label=left:$x_1$](31) at (-1.5*\r,-0.5*\r) {};
     \node (23) at (-1*\r,-2*\r) {}; 
     \node (22) at (-1*\r,-1.5*\r) {};   
     \node (21) at (-1*\r,-0.5*\r) {};
    
     \node(5) at (0*\r,0*\r) {};
     \node(4) at (1*\r,0.5*\r) {};
     \node (3) at (0.5*\r,-1*\r) {};
     \node(1) at (0.5*\r,1*\r) {};
     \node(0) at (1*\r,-0.5*\r) {};

\end{scope}

\begin{scope}
   
    \path [-] (5) edge node {} (4);
    
    \path [-] (1) edge node {} (5);
    \path [-] (5) edge node {} (3);
    
    \path [-] (6) edge node {} (5);
    \path [-] (5) edge node {} (11);
    \path [-] (5) edge node {} (13);
    
	\path [-] (1) edge node {} (4);
    \path [-] (0) edge node {} (5);   
    \path [-] (0) edge node {} (3); 
    
    \path [-] (21) edge node {} (31);
    \path [-] (22) edge node {} (32);   
    \path [-] (23) edge node {} (33);
    
    \path [-] (5) edge node {} (21);
    \path [-] (5) edge node {} (22);   
    \path [-] (5) edge node {} (23);
    
\end{scope}

\path (11) -- node[auto=false]{\vdots} (13);
\path (21) -- node[auto=false]{\vdots} (22);
\end{tikzpicture}}
\end{center}
\caption{The Type 1 bicyclic graph $F(n_1,n_2)$.}\label{fig:F(n1,n2)}
\end{figure}

\begin{lemma}\label{lem:F(n1,n2)lessthanBn}
For all $n\ge 7$, if $n_1$ and $n_2$ are nonnegative integers such that $n_1+2n_2=n-5$, then $\cs(F(n_1,n_2),x)\preceq\cs(B_n,x)$.
\end{lemma}
\begin{proof}
The proof is by induction on $n$. For $n=7$, the result follows as can be checked via the Appendix. Suppose the result holds for all $n_1$ and $n_2$ where $n_1+2n_2=k-5$ for all $7\le k<n$. Now let $n_1$ and $n_2$ be two nonnegative integers such that $n_1+2n_2=n-5$. 

If $n_2=0$, then $\cs(F(n_1,n_2),x)=\cs(F(n-5,0),x)\preceq\cs(B_n,x)$ by Lemma~\ref{lem:explictformulasUnAnCn}.

If $n_2\neq 0$, then let $v$ be a leaf of $F(n_1,n_2)$ with a neighbour of degree $2$. Let $u$ be a leaf of $B_n$. Now $F(n_1,n_2)-v=F(n_1,n_2)/v=F(n_1+1,n_2-1)$ and $B_n-u=B_n/u=B_{n-1}$, so $\cs(F(n_1,n_2)-v,x)\preceq\cs(B_n-u,x)$ and $\cs(F(n_1,n_2)/v,x)\preceq(B_n/u,x)$ by the inductive hypothesis. Since $n\ge 7$, $F(0,0)$ is an induced subgraph of $F(n_1,n_2)-N[v]$, so
$$\cs(B_n-N[u],x)=(n-3)x+2x^2+x^3\preceq  (n-3)x+6x^2+6x^3\preceq \cs(F(n_1,n_2)-N[v],x).$$
Finally, by Lemma~\ref{lem:pivotcomps}, $\cs(F(n_1,n_2),x)\preceq\cs(B_n,x)$.
\end{proof}

\begin{lemma}\label{lem:bicyclicneighborhooddeletion}
Let $G$ be a bicyclic graph of order $n$ with at least one leaf such that $G$ is not isomorphic to $F(n_1,n_2)$ for any nonnegative integers $n_1,n_2$ such that $n_1+2n_2=n-5$. If $v$ is a leaf of $G$ and $u$ is a leaf of $B_n$, then $\cs(B_n-N[u],x)\preceq\cs(G-N[v],x)$.
\end{lemma}
\begin{proof}
Let $G$, $v$, and $u$ be as in the hypothesis. Let $w\in V(G)$ be the unique neighbour of $v$ and consider the graph $G-N[v]$.  It is clear that $G-N[v]$ and $B_n-N[u]$ have the same order and since every connected set of order $3$ induces a graph with at least $2$ edges, it suffices to show that $G-N[v]$ has at least one connected set of order $3$. Suppose $G-N[v]$ has no connected sets of order $3$, i.e. that $G-N[v]=n_2K_2\cup n_1K_1$ for some nonnegative integers $n_1$ and $n_2$ such that $n_1+2n_2=n-2$. Since $G$ is a bicyclic graph, we must have that exactly two of the $K_2$ components are joined to $w$. Now, the $K_1$ components must have been leaves in $G$ adjacent to $w$ since $G$ is a connected graph. Again since $G$ is a connected graph, it follows that the $n_2-2$ remaining $K_2$ components of $G-N[v]$ must have exactly one vertex adjacent to $w$. Therefore, $G=F(n_1+1,n_2-2)$, which contradicts our assumption. Therefore, $G-N[v]$ has at least one connected set of order $3$ and therefore, $\cs(B_n-N[u],x)\preceq\cs(G-N[v],x)$.
\end{proof}   
   
Note that the previous lemma also applies to Type 3 bicyclic graphs and will be used in the next subsection.

\begin{lemma}\label{lem:Type1and2BicyclicSmallerThanBn}
For $n\ge 7$, $\cs(G,x)\preceq \cs(B_n,x)$ for all Type 1 and Type 2 bicyclic graphs $G$.
\end{lemma}
\begin{proof}
The proof is by induction on $n$. For $n=7$, the result follows by direct comparisons as can be checked in the Appendix. Now suppose the result holds for all bicyclic graphs of order $7\le k<n$ and let $G$ be a bicyclic graph of order $n$.\\

\noindent\textbf{Case 1:} $G$ is a Type 1 bicyclic graph and $G=G_1(a,b)$.

Since $n>7$, we may assume without loss of generality that $b\ge 4$. Let $v$ be a vertex of degree $2$ on $C_b$ such that both of its neighbours also have degree $2$. Let $u$ be a leaf of $B_n$. Now, $G-v$ is a unicyclic graph of order $n-1$, not equal to $C_{n-1}$, $G/v=G_1(a,b-1)$, and $G-N[v]$ is a unicyclic graph of order $n-3$. Also, $B_n-u=B_{n-1}$, $B_n/u=B_{n-1}$, and $B_n-N[u]=P_3\cup \overline{K_{n-5}}$. Since $\cs(U_n,x)\preceq \cs(B_n,x)$ from Lemma~\ref{lem:explictformulasUnAnCn} and from Theorem~\ref{thm:UnbiggerthanallexcpetCn}, we have $\cs(G-v,x)\preceq \cs(B_n-u,x)$. From the inductive hypothesis, $\cs(G/v,x)\preceq \cs(B_n/u,x)$. Since $G-N[v]$ is a unicyclic graph of order $n-3\ge 5$, $G$ has at least $5$ edges and a cycle of length at least $3$, therefore $\cs(B_n-N[u],x)\preceq \cs(G-N[v],x)+x$. Now, from a similar argument to the proof of Lemma~\ref{lem:pivotcomps}, it follows that $\cs(G,x)\preceq \cs(B_n,x)+x^2$. However, by definition, $G$ and $B_n$ have the same number of edges, so we can conclude that $\cs(G,x)\preceq \cs(B_n,x)$. \\

\noindent\textbf{Case 2:} $G$ is a Type 1 bicyclic graph and $G\neq G_1(a,b)$.

If $G=F(n_1,n_2)$, then $\cs(G,x)\preceq\cs(B_n,x)$ from Lemma~\ref{lem:F(n1,n2)lessthanBn}. 

If $G\neq F(n_1,n_2)$, then let $v$ be a leaf of $G$. Let $u$ be a leaf of $B_n$. Now, $G-v=G/v$ is a Type 1 bicyclic graph of order $n-1$, so $\cs(G-v,x)\preceq \cs(B_{n-1},x)=\cs(B_n-u,x)$ by the inductive hypothesis. Further, $\cs(B_n-N[u],x)\preceq \cs(G-N[v],x)$ by Lemma~\ref{lem:bicyclicneighborhooddeletion}. Hence Lemma~\ref{lem:pivotcomps} gives $\cs(G,x)\preceq \cs(B_n,x)$.\\

\noindent\textbf{Case 3:} $G$ is a Type 2 bicyclic graph and $G=G_2(a,b,c)$.

If $G=G_2(3,3,n-6)$, then let $v$ be a vertex of degree $2$ on a cycle of $G$. Now, $G-v=G/v$ is a unicyclic graph of order $n-1$ with at least one leaf. Let $u$ be a vertex of degree $2$ in $B_n$. Then $B_n-u=B_n/u=U_{n-1}$, so from Theorem~\ref{thm:UnbiggerthanallexcpetCn}, it follows that 
$$\cs(G-v,x)=\cs(G/v,x)\preceq \cs(U_{n-1},x)=\cs(B_{n}/u,x)=\cs(B_n-u,x). $$
Further, $G-N[v]$ is a unicyclic graph of order $n-3$ and $B_n-N[u]=(n-3)K_1$, so $\cs(B_n-N[u],x)\preceq\cs(G-N[v],x)$. Hence, $\cs(G,x)\preceq\cs(B_n,x)$ by Lemma~\ref{lem:pivotcomps}.

If $G\neq G_2(3,3,n-6)$, then suppose without loss of generality that $a>3$. Let $v$ be a vertex of degree $2$ on the induced $C_a$ of $G$ such that both of its neighbours also have degree $2$. Now, $G/v=G_2(a-1,b,c)$ and therefore, $\cs(G/v,x)\preceq \cs(B_{n-1},x)=\cs(B_n/u,x)$ by the inductive hypothesis where $u$ is again a leaf of $B_n$. We also have $\cs(G-v,x)\preceq\cs(U_{n-1}\preceq \cs(B_n-u,x)$ by Theorem~\ref{thm:UnbiggerthanallexcpetCn} and $\cs(B_n-N[u],x)\preceq \cs(G-N[v],x)+x$. Now, by a similar argument as appeared in Case 1, it follows that $\cs(G,x)\preceq\cs(B_n,x)$.

\noindent\textbf{Case 4:}  $G$ is a Type 2 bicyclic graph and $G\neq G_2(a,b,c)$.

Let $v$ be a leaf of $G$. Let $u$ be a leaf of $B_n$. Now, $G-v=G/v$ is a Type 2 bicyclic graph of order $n-1$, so $\cs(G-v,x)\preceq \cs(B_{n-1},x)=\cs(B_n-u,x)$ by the inductive hypothesis. Further, $\cs(B_n-N[u],x)\preceq \cs(G-N[v],x)$ by Lemma~\ref{lem:bicyclicneighborhooddeletion}. Hence Lemma~\ref{lem:pivotcomps} gives $\cs(G,x)\preceq \cs(B_n,x)$.

\end{proof}

\subsection{Type 3}
%
%

\begin{lemma}\label{lem:G3almostsmallerthanBn}
If $n\ge 8$ and $G_3(a,b,c)$ has order $n$, then $\cs(G_3(a,b,c),x)\preceq \cs(B_n,x)+x^{n-1}$.
\end{lemma}
\begin{proof}
%
We proceed by induction on the order. The result holds for all such graphs of order $8$ (see Table~\ref{tab:G3(a,b,c)order8}). Now suppose that $\cs(G_3(a,b,c),x)\preceq \cs(B_k,x)+x^{k-1}$ for all $a+b+c=k-2$ and $8\le k<n$. Let $a,b,c$ be such that $a+b+c=n-2$ and let $v$ be a vertex of degree $2$ in $G_3(a,b,c)$ such that $v$ is not the only vertex on its subpath (i.e. $G_3(a,b,c)-v\not\cong C_{n-1}$). Note that such a $v$ always exists since $n\ge 9$. Let $u$ be a leaf in $B_n$. Note that $\cs(U_k,x)\preceq \cs(B_k,x)$ for all $k$ as can readily be seen by comparing the formula for each. From this, Theorem~\ref{thm:UnbiggerthanallexcpetCn}, and the fact that $G_3(a,b,c)-v$ is a unicyclic graph of order $n-1$ not isomorphic to $C_{n-1}$, it follows that 
$$\cs(G_3(a,b,c)-v,x)\preceq\cs(U_{n-1},x)\preceq \cs(B_{n-1},x)=\cs(B_n-u,x).$$

Without loss of generality, suppose $G_3(a,b,c)/v\cong G_3(a-1,b,c)$ so, by the inductive hypothesis,
 $$\cs(G_3(a,b,c)/v,x)\preceq\cs(B_{n}/u,x)+x^{n-2}.$$

Since $n>8$ and $v$ is not the only vertex on its cycle, it follows that $G_3(a,b,c)-N[v]$ is a graph of order $n-3$ with at least one connected set of order $3$. Therefore,  $$\cs(B_{n}-N[u],x)\preceq\cs(G_3(a,b,c)-N[v],x)+x.$$

Therefore, as in the proof of Case 1 of Lemma~\ref{lem:Type1and2BicyclicSmallerThanBn}, by a similar argument to Lemma~\ref{lem:pivotcomps}, it follows that $\cs(G_3(a,b,c),x)\preceq\cs(B_{n},x)+x^2+x^{n-1}.$ But $G_3$ and $B_{n}$ have the same number of edges by definition, so we may conclude that

$$\cs(G_3(a,b,c),x)\preceq\cs(B_{n},x)+x^{n-1}.$$

\end{proof}

\begin{table}[!h]
\begin{center}  \small
\renewcommand\arraystretch{1.2}
\begin{tabular}{|c|c|} 
\hline
$G$  &  $\cs(G,x)$    \\ \hline
\scalebox{0.2}{
\begin{tikzpicture}
\GraphInit[vstyle=Classic]
\Vertex[L=\hbox{},x=4.1017cm,y=5.0cm]{v0}
\Vertex[L=\hbox{},x=4.1239cm,y=0.8421cm]{v1}
\Vertex[L=\hbox{},x=1.1351cm,y=1.2068cm]{v2}
\Vertex[L=\hbox{},x=0.0cm,y=0.7446cm]{v3}
\Vertex[L=\hbox{},x=2.044cm,y=4.8411cm]{v4}
\Vertex[L=\hbox{},x=5.0cm,y=2.9381cm]{v5}
\Vertex[L=\hbox{},x=2.0311cm,y=0.0cm]{v6}
\Vertex[L=\hbox{},x=0.5262cm,y=2.8811cm]{v7}
\Edge[](v0)(v4)
\Edge[](v0)(v5)
\Edge[](v1)(v5)
\Edge[](v1)(v6)
\Edge[](v2)(v6)
\Edge[](v2)(v7)
\Edge[](v3)(v6)
\Edge[](v3)(v7)
\Edge[](v4)(v7)
\end{tikzpicture}}  & $x^{8} + 8 \, x^{7} + 18 \, x^{6} + 16 \, x^{5} + 14 \, x^{4} + 12 \, x^{3} + 9 \, x^{2} + 8 \, x$   \\ \hline
\scalebox{0.2}{ 
\begin{tikzpicture}
\GraphInit[vstyle=Classic]
\Vertex[L=\hbox{},x=4.5001cm,y=4.7511cm]{v0}
\Vertex[L=\hbox{},x=4.9022cm,y=0.9221cm]{v1}
\Vertex[L=\hbox{},x=0.0cm,y=1.1374cm]{v2}
\Vertex[L=\hbox{},x=2.4092cm,y=3.1515cm]{v3}
\Vertex[L=\hbox{},x=1.597cm,y=5.0cm]{v4}
\Vertex[L=\hbox{},x=2.3606cm,y=0.0cm]{v5}
\Vertex[L=\hbox{},x=5.0cm,y=2.9348cm]{v6}
\Vertex[L=\hbox{},x=0.1128cm,y=3.2622cm]{v7}
\Edge[](v0)(v4)
\Edge[](v0)(v6)
\Edge[](v1)(v5)
\Edge[](v1)(v6)
\Edge[](v2)(v5)
\Edge[](v2)(v7)
\Edge[](v3)(v6)
\Edge[](v3)(v7)
\Edge[](v4)(v7)
\end{tikzpicture}
}   & $x^{8} + 8 \, x^{7} + 20 \, x^{6} + 21 \, x^{5} + 17 \, x^{4} + 12 \, x^{3} + 9 \, x^{2} + 8 \, x$   \\ \hline
\scalebox{0.2}{\begin{tikzpicture}
\GraphInit[vstyle=Classic]
\Vertex[L=\hbox{},x=2.5909cm,y=3.7457cm]{v0}
\Vertex[L=\hbox{},x=5.0cm,y=0.8482cm]{v1}
\Vertex[L=\hbox{},x=0.0cm,y=3.9894cm]{v2}
\Vertex[L=\hbox{},x=1.5337cm,y=0.8666cm]{v3}
\Vertex[L=\hbox{},x=4.5614cm,y=2.5588cm]{v4}
\Vertex[L=\hbox{},x=1.1557cm,y=5.0cm]{v5}
\Vertex[L=\hbox{},x=3.3393cm,y=0.0cm]{v6}
\Vertex[L=\hbox{},x=1.0666cm,y=2.6238cm]{v7}
\Edge[](v0)(v4)
\Edge[](v0)(v5)
\Edge[](v0)(v7)
\Edge[](v1)(v4)
\Edge[](v1)(v6)
\Edge[](v2)(v5)
\Edge[](v2)(v7)
\Edge[](v3)(v6)
\Edge[](v3)(v7)
\end{tikzpicture}}   & $x^{8} + 8 \, x^{7} + 16 \, x^{6} + 18 \, x^{5} + 15 \, x^{4} + 12 \, x^{3} + 9 \, x^{2} + 8 \, x$   \\ \hline
\scalebox{0.2}{\begin{tikzpicture}
\GraphInit[vstyle=Classic]
\Vertex[L=\hbox{},x=5.0cm,y=1.8004cm]{v0}
\Vertex[L=\hbox{},x=2.1314cm,y=4.8916cm]{v1}
\Vertex[L=\hbox{},x=2.837cm,y=0.0642cm]{v2}
\Vertex[L=\hbox{},x=0.0cm,y=3.2345cm]{v3}
\Vertex[L=\hbox{},x=3.3674cm,y=3.1282cm]{v4}
\Vertex[L=\hbox{},x=4.8362cm,y=0.0cm]{v5}
\Vertex[L=\hbox{},x=0.1524cm,y=5.0cm]{v6}
\Vertex[L=\hbox{},x=1.7675cm,y=1.9334cm]{v7}
\Edge[](v0)(v4)
\Edge[](v0)(v5)
\Edge[](v1)(v4)
\Edge[](v1)(v6)
\Edge[](v2)(v5)
\Edge[](v2)(v7)
\Edge[](v3)(v6)
\Edge[](v3)(v7)
\Edge[](v4)(v7)
\end{tikzpicture}}  & $x^{8} + 8 \, x^{7} + 17 \, x^{6} + 20 \, x^{5} + 18 \, x^{4} + 12 \, x^{3} + 9 \, x^{2} + 8 \, x$   \\ \hline
\scalebox{0.2}{\begin{tikzpicture}
\GraphInit[vstyle=Classic]
\Vertex[L=\hbox{},x=2.816cm,y=5.0cm]{v0}
\Vertex[L=\hbox{},x=0.0cm,y=2.8748cm]{v1}
\Vertex[L=\hbox{},x=5.0cm,y=1.9407cm]{v2}
\Vertex[L=\hbox{},x=1.9864cm,y=0.0cm]{v3}
\Vertex[L=\hbox{},x=1.141cm,y=4.6174cm]{v4}
\Vertex[L=\hbox{},x=4.0763cm,y=3.2641cm]{v5}
\Vertex[L=\hbox{},x=0.3936cm,y=0.6728cm]{v6}
\Vertex[L=\hbox{},x=3.6355cm,y=1.1458cm]{v7}
\Edge[](v0)(v4)
\Edge[](v0)(v5)
\Edge[](v1)(v4)
\Edge[](v1)(v6)
\Edge[](v2)(v5)
\Edge[](v2)(v7)
\Edge[](v3)(v6)
\Edge[](v3)(v7)
\Edge[](v5)(v7)
\end{tikzpicture}}   & $x^{8} + 8 \, x^{7} + 13 \, x^{6} + 12 \, x^{5} + 11 \, x^{4} + 10 \, x^{3} + 9 \, x^{2} + 8 \, x$   \\ \hline
\scalebox{0.2}{\begin{tikzpicture}
\GraphInit[vstyle=Classic]
\Vertex[L=\hbox{},x=1.9012cm,y=0.3052cm]{v0}
\Vertex[L=\hbox{},x=2.4479cm,y=5.0cm]{v1}
\Vertex[L=\hbox{},x=5.0cm,y=1.6649cm]{v2}
\Vertex[L=\hbox{},x=0.0cm,y=0.0cm]{v3}
\Vertex[L=\hbox{},x=0.6229cm,y=4.6579cm]{v4}
\Vertex[L=\hbox{},x=3.2513cm,y=1.0325cm]{v5}
\Vertex[L=\hbox{},x=3.3704cm,y=2.6634cm]{v6}
\Vertex[L=\hbox{},x=0.8256cm,y=2.1761cm]{v7}
\Edge[](v0)(v3)
\Edge[](v0)(v6)
\Edge[](v1)(v4)
\Edge[](v1)(v6)
\Edge[](v2)(v5)
\Edge[](v2)(v6)
\Edge[](v3)(v7)
\Edge[](v4)(v7)
\Edge[](v5)(v7)
\end{tikzpicture}}   & $x^{8} + 8 \, x^{7} + 21 \, x^{6} + 24 \, x^{5} + 17 \, x^{4} + 12 \, x^{3} + 9 \, x^{2} + 8 \, x$   \\ \hline
\end{tabular}
\caption{All graphs of the form $G_3(a,b,c)$ of order $8$.}\label{tab:G3(a,b,c)order8}
\end{center}
\end{table}

\begin{theorem}\label{thm:Type3withleaf}
If $G$ is a Type 3 bicyclic graph of order $n\ge 7$ with at least one leaf, then $\cs(G,x)\preceq \cs(B_n,x)$.
\end{theorem}
\begin{proof}
The proof is by induction on $n$. For $n=7$, the result follows from a computational check (again, see Appendix). Now suppose $G$ is a Type 3 bicyclic graph of order $n\ge 8$ with at least one leaf. Let $v$ be a leaf of $G$ and $u$ be a leaf of $B_n$.

We note that $\cs(B_n-N[u],x)\preceq \cs(G-N[v],x)$ by Lemma~\ref{lem:bicyclicneighborhooddeletion}. 

If $G-v$ again has a leaf, then $\cs(G-v,x)=\cs(G/v,x)\preceq\cs(B_n/u,x)=\cs(B_n-u,x)$ by the inductive hypothesis. Therefore, $\cs(G,x)\preceq \cs(B_n,x)$ by Lemma~\ref{lem:pivotcomps}.

If $G-v$ has no leaves, then by Lemma~\ref{lem:G3almostsmallerthanBn} and a similar argument to Lemma~\ref{lem:pivotcomps}, it follows that $\cs(G,x)\preceq \cs(B_n,x)+x^{n-2}+x^{n-1}$. It is clear that $S_{n-1}(G)=n-1$ as $G$ has exactly one cut-vertex (the vertex adjacent to $v$). It is also clear that the cut-sets of order $2$ correspond exactly to the independent sets of order $2$ of the three paths, $P_{a+2},P_{b+2},$ and $P_{c+2}$ after subtracting $2$ to account for triple counting the independent set common to each. Since a path of order $k$ has $\binom{k}{2}-(k-1)=\frac{(k-1)^2}{2}$ independent sets of order $2$, it follows that $G$ has exactly 

$$\frac{a^2+b^2+c^2+2(a+b+c)-1}{2}$$
cut-sets of order $2$. Since $B_n$ has exactly $n-1$ cut-sets of order $2$, to show that $S_{n-2}(G)\le S_{n-2}(B_n)$, it suffices to show that $$\frac{a^2+b^2+c^2+2(a+b+c)-1}{2}\ge n-1.$$

Since $n\ge 8$, it follows that $a^2+b^2+c^2\ge 3$. By adding $2n-5$ to both sides of this inequality, then dividing both sides by $2$ and substituting $a+b+c=n-2$, we obtain

$$\frac{a^2+b^2+c^2+2(a+b+c)-1}{2}\ge n-1$$
as desired. Therefore, $S_{n-2}(G)\le S_{n-2}(B_n)$, which was the last remaining inequality to conclude that $\cs(G,x)\preceq \cs(B_n,x)$.
\end{proof}

\begin{corollary}
There is no UMR bicyclic graph with respect to node reliability.
\end{corollary}
\begin{proof}
The graph $G_3(a,b,c)$ of order $n$ has no cut-vertices, so $S_{n-1}(G_3(a,b,c))=n>S_{n-1}(B_n)$. Therefore, for values of $p$ close to $1$, $\nrel(G_3(a,b,c),p)>\nrel(B_n,p)$, but the inequality is reversed for values of $p$ close to $0$ from Lemma~\ref{lem:G3almostsmallerthanBn}. Also, from Lemma~\ref{lem:Type1and2BicyclicSmallerThanBn} and Theorem~\ref{thm:Type3withleaf}, it follows that $\nrel(G,p)\le \nrel(B_n,p)$ for all $p\in [0,1]$ and all bicyclic graphs $G$ with $G\neq G_2(a,b,c)$. Therefore, there is no UMR bicyclic graph with respect to node reliability.
\end{proof}

We are now ready to move from node reliability to node cop-win reliability.
 
\begin{lemma}\label{lem:unicyclicnotcopwin}
If $G$ is a unicyclic graph with an induced cycle $C$ with $|V(C)|\ge 4$, then $G$ is not cop-win
\end{lemma} 
\begin{proof}
Suppose that $G$ has an induced cycle $C$ with vertices labelled around the cycle $u_1,u_2,\ldots u_k$ for some $k\ge 4$ but that $G$ is cop-win. Let the robber move only on the vertices of $C$. Since the cop can catch the robber, there must be a vertex $v\in V(G)-V(C)$ such that $v$ is adjacent to at least three consecutive vertices on the cycle, without loss of generality say $u_1,u_2,u_3$. But now, $\{u_1,u_2,u_3,v\}$ induces a connected subgraph of $G$ with $4$ vertices and $5$ edges, contradicting $G$ being a unicyclic graph. Therefore, $G$ is not cop-win.
\end{proof}

\begin{theorem}
If $G$ is a bicyclic graph of order $n\ge 7$, then $\cw(G,x)\preceq \cw(B_n,x)$.
\end{theorem}
\begin{proof}
If $G$ is a Type 1 or Type 2 bicyclic graph, then the result follows from Lemma~\ref{lem:Type1and2BicyclicSmallerThanBn}. If $G$ is of Type 3 and has a leaf, then the result follows from Theorem~\ref{thm:Type3withleaf}.

Otherwise, $G=G_3(a,b,c)$. If $n=7$, then the result follows by checking the Appendix. For $n\ge 8$, it follows from Lemma~\ref{lem:G3almostsmallerthanBn} that $\cs(G,x)\preceq\cs(B_n,x)+x^{n-1}$.  Now, since $n\ge 8$, it follows that at least one of $a$, $b$, or $c$ is at least $2$. Without loss of generality, suppose $a\ge 2$. Further, since $G$ is bicyclic it follows that at most one of $b$ and $c$ are equal to $0$. Without loss of generality, suppose $b>0$. Let $v$ be a vertex of $G$ on the induced $P_b$. It is clear that $G-v$ is a unicyclic graph and therefore connected. Also, in $G-v$, $P_a$ together with the two end vertices induce a cycle $C$ with $|V(C)|\ge 4$. Therefore, by Lemma~\ref{lem:unicyclicnotcopwin}, it follows that $G-v$ is not cop-win. Therefore, $W_{n-1}(G)\le S_{n-1}(G)-1$ and so, by Lemma~\ref{lem:G3almostsmallerthanBn}, $\cw(G,x)\preceq\cs(G,x)-x^{n-1}\preceq\cs(B_n,x)=\cw(B_n,x)$. 

\end{proof}

\begin{corollary}
For $n\ge 7$, $B_n$ is UMR with respect to node cop-win reliability among all bicyclic graphs of order $n$.
\end{corollary}

\section{Conclusion}\label{sec:conclusion}
We have shown that $U_n$ and $B_n$ are UMR graphs with respect to node cop-win reliability. As we have only introduced node cop-win reliability in this paper, directions for future research are wide-open. We list a few possibilities here.

The most natural extension of our work would be to study UMR graphs with respect to node cop-win reliability for other families of sparse graphs. A graph is \textit{$m$-cyclic} if it is connected and has order $n$ and $n-1+m$ edges. For $m=0,1,2$, the UMR graphs are $K_{1,n}$, $U_n$, and $B_n$ respectively. These graphs are clearly very closely related and can be extended to the graphs $H_{n,m}$ where $H_{n,0}=K_{1,n-1}$ with $v\in V(K_{1,n-1})$ a leaf. For $m>0$, let $H_{n,m}$ be obtained from $H_{n,m-1}$ by adding an edge between a leaf and $v$ (see Figure~\ref{fig:Hnm}). Computational evidence leads us to make the following conjecture.

\begin{figure}
\def\c{0.5}
\def\r{2}
\begin{center}
\scalebox{\c}{
\begin{tikzpicture}
\begin{scope}[every node/.style={circle,thick,draw,fill}]

	 \node[label=below:$u_{m}$] (7) at (1*\r,-1*\r) {};
     \node[label=below:$u_2$] (6) at (-0.5*\r,-1*\r) {};
     \node[label=below:$u_1$] (5) at (-1*\r,-1*\r) {};
    
     \node[label=above right:$v_{n-2-m}$] (4) at (1*\r,2*\r) {};
     \node[label=above:$v_2$] (3) at (-0.5*\r,2*\r) {};
     \node[label=above:$v_1$] (2) at (-1*\r,2*\r) {};
     \node (1) at (0*\r,1*\r) {};
     \node[label=right:$v$] (0) at (0*\r,0*\r) {};

\end{scope}

\begin{scope}
   
    \path [-] (0) edge node {} (1);
    
    \path [-] (1) edge node {} (2);
    \path [-] (1) edge node {} (3);
    \path [-] (1) edge node {} (4);
    
    \path [-] (1) edge node {} (7);
    \path [-] (1) edge node {} (6);
    \path [-] (1) edge node {} (5);
    \path [-] (0) edge node {} (7);
    \path [-] (0) edge node {} (6);
    \path [-] (0) edge node {} (5);

\end{scope}

\path (3) -- node[auto=false]{$\cdots$} (4);
\path (6) -- node[auto=false]{$\cdots$} (7);
\end{tikzpicture}}
\end{center}
\caption{The graph $H_{n,m}$.}\label{fig:Hnm}
\end{figure}
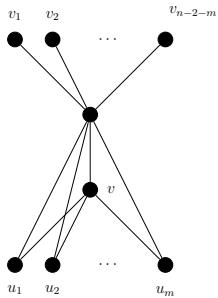

\begin{conjecture}
Among all $m$-cyclic graphs, $H_{n,m}$ is UMR with respect to node cop-win reliability for all $n$ sufficiently large.
\end{conjecture}


Another direction would be to consider a generalization of node cop-win reliability to \textit{node $k$-cop-win reliability} that is defined analogously as the probability that $k$ cops can guarantee a win on the graph. 

\begin{problem}
Are there any families of graphs that have a UMR graph with respect to node $k$-cop-win reliability for any $k>1$?
\end{problem}

Our work was motivated by the lack of UMR graphs with respect to node reliability, but this could of course be reframed in a model where edges fail. Let $G$ be a graph with perfectly reliable nodes whose edges fail independently with probability $p\in [0,1]$. The \textit{cop-win reliability} of a graph then is defined as the probability that the graph is cop-win. 

\begin{problem}
Are there any families of graphs that have a UMR graph with respect to cop-win reliability?
\end{problem}

As our work hinged on polynomials, it is natural to wonder about their analytic properties. For the node cop-win reliability, many questions have already been answered via work done on the node reliability of chordal graphs. For example, it follows from \cite{BrownMol2016} that the set $\{z\in \mathbb{C}:\ \ncrel(G,z)=0\text{ for some graph $G$}\}$ is dense in $\mathbb{C}$. However, the analytic properties of the roots of cop-win reliability could make for very interesting study given to the fascinating behaviour of the roots of \textit{all-terminal reliability} (the probability the graph is connected given random edge failures). It was conjectured that the roots of all-terminal reliability were contained in the disk $|z-1|\le 1$ \cite{BrownColbourn1992} but was eventually shown to be false \cite{RoyleSokal2004}. It is interesting to note however, that the largest known counterexample to the Brown-Colbourn Conjecture is only outside of this disk by $0.1134860896$ \cite{BrownMol2017}. From computations on graphs of small order, it appears that this behaviour continues for roots of cop-win reliability as all roots are contained in the disk $|z-1|\le 1$.

\begin{problem}
What can be said about the location of the roots of cop-win reliability in the complex plane?
\end{problem}

\section*{Acknowledgements}
The authors thank Lucas Mol for helpful suggestions on a preliminary version of this work.

\bibliographystyle{abbrv}
\bibliography{References}

\newpage

\section*{Appendix}
The following table contains the edge sets and connected set polynomials for all bicyclic graphs of order $7$ with vertex set $\{0,1,2,3,4,5,6\}$.
\begin{center}  \scriptsize
\renewcommand\arraystretch{1.2}
\begin{longtable}{|c|c|} 
\hline
$E(G)$ & $\cs(G,x)$ \\ \hline
$ \{\{0, 5\}, \{0, 6\}, \{1, 5\}, \{1, 6\}, \{2, 6\}, \{3, 6\}, \{4, 6\}, \{5, 6\}\} $ & $ x^{7} + 6 \, x^{6} + 15 \, x^{5} + 20 \, x^{4} + 16 \, x^{3} + 8 \, x^{2} + 7 \, x $  \\ \hline
$ \{\{0, 5\}, \{0, 6\}, \{1, 5\}, \{1, 6\}, \{2, 5\}, \{2, 6\}, \{3, 6\}, \{4, 6\}\} $ & $ x^{7} + 6 \, x^{6} + 15 \, x^{5} + 20 \, x^{4} + 16 \, x^{3} + 8 \, x^{2} + 7 \, x $  \\ \hline
$ \{\{0, 5\}, \{0, 6\}, \{1, 5\}, \{1, 6\}, \{2, 5\}, \{3, 6\}, \{4, 6\}, \{5, 6\}\} $ & $ x^{7} + 5 \, x^{6} + 11 \, x^{5} + 15 \, x^{4} + 14 \, x^{3} + 8 \, x^{2} + 7 \, x $  \\ \hline
$ \{\{0, 5\}, \{0, 6\}, \{1, 5\}, \{1, 6\}, \{2, 5\}, \{2, 6\}, \{3, 5\}, \{4, 6\}\} $ & $ x^{7} + 5 \, x^{6} + 12 \, x^{5} + 17 \, x^{4} + 15 \, x^{3} + 8 \, x^{2} + 7 \, x $  \\ \hline
$ \{\{0, 4\}, \{0, 6\}, \{1, 5\}, \{1, 6\}, \{2, 6\}, \{3, 6\}, \{4, 6\}, \{5, 6\}\} $ & $ x^{7} + 6 \, x^{6} + 15 \, x^{5} + 20 \, x^{4} + 15 \, x^{3} + 8 \, x^{2} + 7 \, x $  \\ \hline
$ \{\{0, 4\}, \{0, 5\}, \{0, 6\}, \{1, 5\}, \{1, 6\}, \{2, 6\}, \{3, 6\}, \{4, 6\}\} $ & $ x^{7} + 6 \, x^{6} + 14 \, x^{5} + 18 \, x^{4} + 14 \, x^{3} + 8 \, x^{2} + 7 \, x $  \\ \hline
$ \{\{0, 4\}, \{0, 5\}, \{0, 6\}, \{1, 5\}, \{1, 6\}, \{2, 6\}, \{3, 6\}, \{5, 6\}\} $ & $ x^{7} + 5 \, x^{6} + 11 \, x^{5} + 15 \, x^{4} + 13 \, x^{3} + 8 \, x^{2} + 7 \, x $  \\ \hline
$ \{\{0, 4\}, \{0, 6\}, \{1, 5\}, \{1, 6\}, \{2, 5\}, \{2, 6\}, \{3, 6\}, \{4, 6\}\} $ & $ x^{7} + 6 \, x^{6} + 14 \, x^{5} + 17 \, x^{4} + 13 \, x^{3} + 8 \, x^{2} + 7 \, x $  \\ \hline
$ \{\{0, 4\}, \{0, 6\}, \{1, 5\}, \{1, 6\}, \{2, 5\}, \{2, 6\}, \{3, 6\}, \{5, 6\}\} $ & $ x^{7} + 5 \, x^{6} + 11 \, x^{5} + 14 \, x^{4} + 12 \, x^{3} + 8 \, x^{2} + 7 \, x $  \\ \hline
$ \{\{0, 4\}, \{0, 6\}, \{1, 5\}, \{1, 6\}, \{2, 5\}, \{3, 6\}, \{4, 6\}, \{5, 6\}\} $ & $ x^{7} + 5 \, x^{6} + 11 \, x^{5} + 14 \, x^{4} + 12 \, x^{3} + 8 \, x^{2} + 7 \, x $  \\ \hline
$ \{\{0, 4\}, \{0, 5\}, \{0, 6\}, \{1, 6\}, \{2, 6\}, \{3, 6\}, \{4, 5\}, \{4, 6\}\} $ & $ x^{7} + 6 \, x^{6} + 14 \, x^{5} + 17 \, x^{4} + 13 \, x^{3} + 8 \, x^{2} + 7 \, x $  \\ \hline
$ \{\{0, 4\}, \{0, 5\}, \{0, 6\}, \{1, 5\}, \{1, 6\}, \{2, 5\}, \{2, 6\}, \{3, 6\}\} $ & $ x^{7} + 5 \, x^{6} + 12 \, x^{5} + 17 \, x^{4} + 14 \, x^{3} + 8 \, x^{2} + 7 \, x $  \\ \hline
$ \{\{0, 4\}, \{0, 5\}, \{0, 6\}, \{1, 5\}, \{1, 6\}, \{2, 5\}, \{3, 6\}, \{4, 6\}\} $ & $ x^{7} + 5 \, x^{6} + 11 \, x^{5} + 14 \, x^{4} + 12 \, x^{3} + 8 \, x^{2} + 7 \, x $  \\ \hline
$ \{\{0, 4\}, \{0, 5\}, \{1, 5\}, \{1, 6\}, \{2, 5\}, \{2, 6\}, \{3, 6\}, \{4, 6\}\} $ & $ x^{7} + 6 \, x^{6} + 15 \, x^{5} + 18 \, x^{4} + 13 \, x^{3} + 8 \, x^{2} + 7 \, x $  \\ \hline
$ \{\{0, 4\}, \{0, 5\}, \{0, 6\}, \{1, 5\}, \{1, 6\}, \{2, 5\}, \{3, 6\}, \{5, 6\}\} $ & $ x^{7} + 4 \, x^{6} + 9 \, x^{5} + 13 \, x^{4} + 12 \, x^{3} + 8 \, x^{2} + 7 \, x $  \\ \hline
$ \{\{0, 4\}, \{0, 5\}, \{0, 6\}, \{1, 5\}, \{2, 5\}, \{3, 6\}, \{4, 6\}, \{5, 6\}\} $ & $ x^{7} + 5 \, x^{6} + 11 \, x^{5} + 14 \, x^{4} + 12 \, x^{3} + 8 \, x^{2} + 7 \, x $  \\ \hline
$ \{\{0, 4\}, \{0, 5\}, \{1, 5\}, \{1, 6\}, \{2, 5\}, \{3, 6\}, \{4, 6\}, \{5, 6\}\} $ & $ x^{7} + 5 \, x^{6} + 12 \, x^{5} + 16 \, x^{4} + 13 \, x^{3} + 8 \, x^{2} + 7 \, x $  \\ \hline
$ \{\{0, 4\}, \{0, 6\}, \{1, 5\}, \{1, 6\}, \{2, 5\}, \{2, 6\}, \{3, 5\}, \{3, 6\}\} $ & $ x^{7} + 5 \, x^{6} + 11 \, x^{5} + 14 \, x^{4} + 13 \, x^{3} + 8 \, x^{2} + 7 \, x $  \\ \hline
$ \{\{0, 4\}, \{0, 6\}, \{1, 5\}, \{1, 6\}, \{2, 5\}, \{2, 6\}, \{3, 5\}, \{4, 6\}\} $ & $ x^{7} + 5 \, x^{6} + 10 \, x^{5} + 12 \, x^{4} + 11 \, x^{3} + 8 \, x^{2} + 7 \, x $  \\ \hline
$ \{\{0, 4\}, \{0, 6\}, \{1, 5\}, \{1, 6\}, \{2, 5\}, \{2, 6\}, \{3, 5\}, \{5, 6\}\} $ & $ x^{7} + 4 \, x^{6} + 8 \, x^{5} + 11 \, x^{4} + 11 \, x^{3} + 8 \, x^{2} + 7 \, x $  \\ \hline
$ \{\{0, 4\}, \{0, 6\}, \{1, 5\}, \{1, 6\}, \{2, 5\}, \{3, 5\}, \{4, 6\}, \{5, 6\}\} $ & $ x^{7} + 5 \, x^{6} + 10 \, x^{5} + 12 \, x^{4} + 11 \, x^{3} + 8 \, x^{2} + 7 \, x $  \\ \hline
$ \{\{0, 4\}, \{0, 5\}, \{0, 6\}, \{1, 5\}, \{1, 6\}, \{2, 6\}, \{3, 6\}, \{4, 5\}\} $ & $ x^{7} + 6 \, x^{6} + 13 \, x^{5} + 15 \, x^{4} + 12 \, x^{3} + 8 \, x^{2} + 7 \, x $  \\ \hline
$ \{\{0, 4\}, \{0, 5\}, \{0, 6\}, \{1, 5\}, \{2, 6\}, \{3, 6\}, \{4, 5\}, \{4, 6\}\} $ & $ x^{7} + 5 \, x^{6} + 10 \, x^{5} + 12 \, x^{4} + 11 \, x^{3} + 8 \, x^{2} + 7 \, x $  \\ \hline
$ \{\{0, 4\}, \{0, 5\}, \{0, 6\}, \{1, 4\}, \{1, 5\}, \{1, 6\}, \{2, 6\}, \{3, 6\}\} $ & $ x^{7} + 6 \, x^{6} + 14 \, x^{5} + 17 \, x^{4} + 14 \, x^{3} + 8 \, x^{2} + 7 \, x $  \\ \hline
$ \{\{0, 4\}, \{0, 5\}, \{0, 6\}, \{1, 4\}, \{1, 5\}, \{2, 6\}, \{3, 6\}, \{4, 6\}\} $ & $ x^{7} + 6 \, x^{6} + 12 \, x^{5} + 13 \, x^{4} + 12 \, x^{3} + 8 \, x^{2} + 7 \, x $  \\ \hline
$ \{\{0, 4\}, \{0, 5\}, \{0, 6\}, \{1, 4\}, \{1, 6\}, \{2, 5\}, \{2, 6\}, \{3, 6\}\} $ & $ x^{7} + 6 \, x^{6} + 14 \, x^{5} + 17 \, x^{4} + 13 \, x^{3} + 8 \, x^{2} + 7 \, x $  \\ \hline
$ \{\{0, 4\}, \{0, 5\}, \{0, 6\}, \{1, 4\}, \{1, 6\}, \{2, 5\}, \{3, 6\}, \{4, 6\}\} $ & $ x^{7} + 4 \, x^{6} + 8 \, x^{5} + 10 \, x^{4} + 10 \, x^{3} + 8 \, x^{2} + 7 \, x $  \\ \hline
$ \{\{0, 4\}, \{0, 5\}, \{0, 6\}, \{1, 4\}, \{2, 5\}, \{2, 6\}, \{3, 6\}, \{4, 6\}\} $ & $ x^{7} + 5 \, x^{6} + 11 \, x^{5} + 14 \, x^{4} + 12 \, x^{3} + 8 \, x^{2} + 7 \, x $  \\ \hline
$ \{\{0, 4\}, \{0, 5\}, \{1, 4\}, \{1, 6\}, \{2, 5\}, \{2, 6\}, \{3, 6\}, \{4, 6\}\} $ & $ x^{7} + 6 \, x^{6} + 13 \, x^{5} + 14 \, x^{4} + 11 \, x^{3} + 8 \, x^{2} + 7 \, x $  \\ \hline
$ \{\{0, 4\}, \{0, 5\}, \{0, 6\}, \{1, 4\}, \{2, 5\}, \{3, 6\}, \{4, 6\}, \{5, 6\}\} $ & $ x^{7} + 4 \, x^{6} + 9 \, x^{5} + 12 \, x^{4} + 11 \, x^{3} + 8 \, x^{2} + 7 \, x $  \\ \hline
$ \{\{0, 4\}, \{0, 5\}, \{1, 4\}, \{1, 6\}, \{2, 5\}, \{3, 6\}, \{4, 6\}, \{5, 6\}\} $ & $ x^{7} + 5 \, x^{6} + 12 \, x^{5} + 15 \, x^{4} + 12 \, x^{3} + 8 \, x^{2} + 7 \, x $  \\ \hline
$ \{\{0, 4\}, \{0, 6\}, \{1, 4\}, \{1, 6\}, \{2, 5\}, \{2, 6\}, \{3, 5\}, \{3, 6\}\} $ & $ x^{7} + 6 \, x^{6} + 13 \, x^{5} + 14 \, x^{4} + 12 \, x^{3} + 8 \, x^{2} + 7 \, x $  \\ \hline
$ \{\{0, 4\}, \{0, 6\}, \{1, 4\}, \{1, 6\}, \{2, 5\}, \{2, 6\}, \{3, 5\}, \{4, 6\}\} $ & $ x^{7} + 4 \, x^{6} + 7 \, x^{5} + 8 \, x^{4} + 9 \, x^{3} + 8 \, x^{2} + 7 \, x $  \\ \hline
$ \{\{0, 4\}, \{0, 6\}, \{1, 4\}, \{2, 5\}, \{2, 6\}, \{3, 5\}, \{3, 6\}, \{4, 6\}\} $ & $ x^{7} + 5 \, x^{6} + 10 \, x^{5} + 12 \, x^{4} + 11 \, x^{3} + 8 \, x^{2} + 7 \, x $  \\ \hline
$ \{\{0, 4\}, \{0, 6\}, \{1, 4\}, \{1, 6\}, \{2, 5\}, \{3, 5\}, \{4, 6\}, \{5, 6\}\} $ & $ x^{7} + 5 \, x^{6} + 10 \, x^{5} + 11 \, x^{4} + 10 \, x^{3} + 8 \, x^{2} + 7 \, x $  \\ \hline
$ \{\{0, 4\}, \{0, 6\}, \{1, 4\}, \{2, 5\}, \{2, 6\}, \{3, 5\}, \{4, 6\}, \{5, 6\}\} $ & $ x^{7} + 4 \, x^{6} + 8 \, x^{5} + 10 \, x^{4} + 10 \, x^{3} + 8 \, x^{2} + 7 \, x $  \\ \hline
$ \{\{0, 4\}, \{0, 5\}, \{0, 6\}, \{1, 4\}, \{1, 5\}, \{1, 6\}, \{2, 5\}, \{3, 6\}\} $ & $ x^{7} + 5 \, x^{6} + 11 \, x^{5} + 15 \, x^{4} + 13 \, x^{3} + 8 \, x^{2} + 7 \, x $  \\ \hline
$ \{\{0, 4\}, \{0, 5\}, \{1, 4\}, \{1, 6\}, \{2, 5\}, \{2, 6\}, \{3, 5\}, \{3, 6\}\} $ & $ x^{7} + 7 \, x^{6} + 15 \, x^{5} + 13 \, x^{4} + 11 \, x^{3} + 8 \, x^{2} + 7 \, x $  \\ \hline
$ \{\{0, 4\}, \{0, 5\}, \{1, 4\}, \{1, 6\}, \{2, 5\}, \{2, 6\}, \{3, 5\}, \{4, 6\}\} $ & $ x^{7} + 6 \, x^{6} + 12 \, x^{5} + 12 \, x^{4} + 10 \, x^{3} + 8 \, x^{2} + 7 \, x $  \\ \hline
$ \{\{0, 4\}, \{0, 5\}, \{1, 4\}, \{2, 5\}, \{2, 6\}, \{3, 5\}, \{3, 6\}, \{4, 6\}\} $ & $ x^{7} + 6 \, x^{6} + 14 \, x^{5} + 16 \, x^{4} + 12 \, x^{3} + 8 \, x^{2} + 7 \, x $  \\ \hline
$ \{\{0, 3\}, \{0, 6\}, \{1, 4\}, \{1, 6\}, \{2, 5\}, \{2, 6\}, \{3, 6\}, \{4, 6\}\} $ & $ x^{7} + 5 \, x^{6} + 11 \, x^{5} + 14 \, x^{4} + 11 \, x^{3} + 8 \, x^{2} + 7 \, x $  \\ \hline
$ \{\{0, 3\}, \{0, 5\}, \{0, 6\}, \{1, 4\}, \{1, 6\}, \{2, 5\}, \{2, 6\}, \{3, 6\}\} $ & $ x^{7} + 5 \, x^{6} + 10 \, x^{5} + 13 \, x^{4} + 11 \, x^{3} + 8 \, x^{2} + 7 \, x $  \\ \hline
$ \{\{0, 3\}, \{0, 5\}, \{0, 6\}, \{1, 4\}, \{1, 6\}, \{2, 5\}, \{2, 6\}, \{4, 6\}\} $ & $ x^{7} + 5 \, x^{6} + 11 \, x^{5} + 14 \, x^{4} + 11 \, x^{3} + 8 \, x^{2} + 7 \, x $  \\ \hline
$ \{\{0, 3\}, \{0, 5\}, \{0, 6\}, \{1, 4\}, \{1, 6\}, \{2, 5\}, \{3, 6\}, \{4, 6\}\} $ & $ x^{7} + 4 \, x^{6} + 7 \, x^{5} + 9 \, x^{4} + 9 \, x^{3} + 8 \, x^{2} + 7 \, x $  \\ \hline
$ \{\{0, 3\}, \{0, 5\}, \{1, 4\}, \{1, 6\}, \{2, 5\}, \{2, 6\}, \{3, 6\}, \{4, 6\}\} $ & $ x^{7} + 6 \, x^{6} + 12 \, x^{5} + 13 \, x^{4} + 10 \, x^{3} + 8 \, x^{2} + 7 \, x $  \\ \hline
$ \{\{0, 3\}, \{0, 5\}, \{0, 6\}, \{1, 4\}, \{1, 6\}, \{2, 5\}, \{2, 6\}, \{5, 6\}\} $ & $ x^{7} + 4 \, x^{6} + 8 \, x^{5} + 11 \, x^{4} + 10 \, x^{3} + 8 \, x^{2} + 7 \, x $  \\ \hline
$ \{\{0, 3\}, \{0, 5\}, \{0, 6\}, \{1, 4\}, \{1, 6\}, \{2, 5\}, \{4, 6\}, \{5, 6\}\} $ & $ x^{7} + 4 \, x^{6} + 8 \, x^{5} + 11 \, x^{4} + 10 \, x^{3} + 8 \, x^{2} + 7 \, x $  \\ \hline
$ \{\{0, 3\}, \{0, 5\}, \{0, 6\}, \{1, 4\}, \{1, 6\}, \{2, 6\}, \{3, 5\}, \{3, 6\}\} $ & $ x^{7} + 5 \, x^{6} + 10 \, x^{5} + 12 \, x^{4} + 10 \, x^{3} + 8 \, x^{2} + 7 \, x $  \\ \hline
$ \{\{0, 3\}, \{0, 5\}, \{0, 6\}, \{1, 4\}, \{1, 6\}, \{2, 6\}, \{3, 5\}, \{4, 6\}\} $ & $ x^{7} + 5 \, x^{6} + 10 \, x^{5} + 11 \, x^{4} + 9 \, x^{3} + 8 \, x^{2} + 7 \, x $  \\ \hline
$ \{\{0, 3\}, \{0, 5\}, \{1, 4\}, \{1, 5\}, \{2, 5\}, \{2, 6\}, \{3, 6\}, \{4, 6\}\} $ & $ x^{7} + 7 \, x^{6} + 16 \, x^{5} + 16 \, x^{4} + 11 \, x^{3} + 8 \, x^{2} + 7 \, x $  \\ \hline
$ \{\{0, 3\}, \{0, 5\}, \{0, 6\}, \{1, 4\}, \{1, 5\}, \{1, 6\}, \{2, 6\}, \{3, 5\}\} $ & $ x^{7} + 5 \, x^{6} + 11 \, x^{5} + 14 \, x^{4} + 11 \, x^{3} + 8 \, x^{2} + 7 \, x $  \\ \hline
$ \{\{0, 3\}, \{0, 5\}, \{1, 4\}, \{1, 6\}, \{2, 5\}, \{2, 6\}, \{3, 5\}, \{4, 6\}\} $ & $ x^{7} + 4 \, x^{6} + 6 \, x^{5} + 6 \, x^{4} + 7 \, x^{3} + 8 \, x^{2} + 7 \, x $  \\ \hline
$ \{\{0, 3\}, \{0, 4\}, \{0, 6\}, \{1, 4\}, \{1, 5\}, \{1, 6\}, \{2, 5\}, \{2, 6\}\} $ & $ x^{7} + 6 \, x^{6} + 12 \, x^{5} + 14 \, x^{4} + 12 \, x^{3} + 8 \, x^{2} + 7 \, x $  \\ \hline
$ \{\{0, 3\}, \{0, 4\}, \{0, 6\}, \{1, 4\}, \{1, 5\}, \{1, 6\}, \{2, 5\}, \{3, 6\}\} $ & $ x^{7} + 5 \, x^{6} + 9 \, x^{5} + 11 \, x^{4} + 10 \, x^{3} + 8 \, x^{2} + 7 \, x $  \\ \hline
$ \{\{0, 3\}, \{0, 4\}, \{0, 6\}, \{1, 4\}, \{1, 5\}, \{2, 5\}, \{2, 6\}, \{3, 6\}\} $ & $ x^{7} + 7 \, x^{6} + 11 \, x^{5} + 10 \, x^{4} + 9 \, x^{3} + 8 \, x^{2} + 7 \, x $  \\ \hline
$ \{\{0, 3\}, \{0, 4\}, \{1, 4\}, \{1, 5\}, \{1, 6\}, \{2, 5\}, \{2, 6\}, \{3, 6\}\} $ & $ x^{7} + 7 \, x^{6} + 13 \, x^{5} + 14 \, x^{4} + 11 \, x^{3} + 8 \, x^{2} + 7 \, x $  \\ \hline
$ \{\{0, 3\}, \{0, 4\}, \{0, 6\}, \{1, 4\}, \{1, 5\}, \{1, 6\}, \{2, 5\}, \{4, 6\}\} $ & $ x^{7} + 4 \, x^{6} + 7 \, x^{5} + 9 \, x^{4} + 9 \, x^{3} + 8 \, x^{2} + 7 \, x $  \\ \hline
$ \{\{0, 3\}, \{0, 4\}, \{0, 6\}, \{1, 4\}, \{1, 5\}, \{2, 5\}, \{2, 6\}, \{4, 6\}\} $ & $ x^{7} + 6 \, x^{6} + 10 \, x^{5} + 12 \, x^{4} + 10 \, x^{3} + 8 \, x^{2} + 7 \, x $  \\ \hline
$ \{\{0, 3\}, \{0, 4\}, \{1, 4\}, \{1, 5\}, \{1, 6\}, \{2, 5\}, \{2, 6\}, \{4, 6\}\} $ & $ x^{7} + 5 \, x^{6} + 8 \, x^{5} + 10 \, x^{4} + 10 \, x^{3} + 8 \, x^{2} + 7 \, x $  \\ \hline
$ \{\{0, 3\}, \{0, 4\}, \{0, 6\}, \{1, 4\}, \{1, 5\}, \{1, 6\}, \{2, 5\}, \{5, 6\}\} $ & $ x^{7} + 5 \, x^{6} + 10 \, x^{5} + 12 \, x^{4} + 11 \, x^{3} + 8 \, x^{2} + 7 \, x $  \\ \hline
$ \{\{0, 3\}, \{0, 4\}, \{0, 6\}, \{1, 4\}, \{1, 5\}, \{2, 5\}, \{2, 6\}, \{5, 6\}\} $ & $ x^{7} + 6 \, x^{6} + 12 \, x^{5} + 13 \, x^{4} + 10 \, x^{3} + 8 \, x^{2} + 7 \, x $  \\ \hline
$ \{\{0, 3\}, \{0, 4\}, \{0, 6\}, \{1, 4\}, \{1, 5\}, \{1, 6\}, \{2, 6\}, \{3, 5\}\} $ & $ x^{7} + 6 \, x^{6} + 13 \, x^{5} + 17 \, x^{4} + 12 \, x^{3} + 8 \, x^{2} + 7 \, x $  \\ \hline
$ \{\{0, 3\}, \{0, 6\}, \{1, 4\}, \{1, 5\}, \{1, 6\}, \{2, 4\}, \{2, 5\}, \{2, 6\}\} $ & $ x^{7} + 5 \, x^{6} + 10 \, x^{5} + 12 \, x^{4} + 12 \, x^{3} + 8 \, x^{2} + 7 \, x $  \\ \hline
$ \{\{0, 3\}, \{0, 6\}, \{1, 4\}, \{1, 5\}, \{1, 6\}, \{2, 4\}, \{2, 5\}, \{3, 6\}\} $ & $ x^{7} + 5 \, x^{6} + 9 \, x^{5} + 9 \, x^{4} + 9 \, x^{3} + 8 \, x^{2} + 7 \, x $  \\ \hline
$ \{\{0, 3\}, \{0, 4\}, \{0, 6\}, \{1, 5\}, \{1, 6\}, \{2, 5\}, \{3, 4\}, \{3, 6\}\} $ & $ x^{7} + 4 \, x^{6} + 6 \, x^{5} + 7 \, x^{4} + 8 \, x^{3} + 8 \, x^{2} + 7 \, x $  \\ \hline
$ \{\{0, 3\}, \{0, 4\}, \{0, 6\}, \{1, 5\}, \{1, 6\}, \{2, 5\}, \{3, 4\}, \{5, 6\}\} $ & $ x^{7} + 4 \, x^{6} + 7 \, x^{5} + 8 \, x^{4} + 8 \, x^{3} + 8 \, x^{2} + 7 \, x $  \\ \hline
$ \{\{0, 3\}, \{0, 4\}, \{0, 6\}, \{1, 5\}, \{2, 5\}, \{3, 4\}, \{3, 6\}, \{5, 6\}\} $ & $ x^{7} + 5 \, x^{6} + 9 \, x^{5} + 9 \, x^{4} + 9 \, x^{3} + 8 \, x^{2} + 7 \, x $  \\ \hline
\caption{The edge set and connected set polynomial for all graphs of order $7$.}\label{tab:CSorder7}
\end{longtable}

\end{center}

\end{document}